\documentclass[12pt,reqno]{article}

\usepackage{amsfonts}
\usepackage{color}
\usepackage{amsmath}
\usepackage{amsthm}
\usepackage{amssymb}
\usepackage{mathrsfs}
\usepackage{amstext}
\usepackage{graphicx}
\usepackage{tikz}
\numberwithin{equation}{section}

\theoremstyle{plain}
\newtheorem{satz}{Theorem}[section]
\newtheorem{defi}[satz]{Definition}
\newtheorem{cor}[satz]{Corollary}\newtheorem{lem}[satz]{Lemma}
\newtheorem{prop}[satz]{Proposition}
\newtheorem{rem}[satz]{Remark}

\newcommand{\re}{\ensuremath{\mathbb{R}}}\newcommand{\N}{\ensuremath{\mathbb{N}}}
\newcommand{\zz}{\ensuremath{\mathbb{Z}}}\newcommand{\C}{\ensuremath{\mathbb{C}}}
\newcommand{\T}{\ensuremath{\mathbb{T}^d}}\newcommand{\tor}{\ensuremath{\mathbb{T}}}

\newcommand{\Z}{{\ensuremath{\zz}^d}}

\newcommand{\R}{\ensuremath{{\re}^d}}

\newcommand{\rank}{{\rm rank \, }}

\newcommand{\mix}{{\rm mix}}
\newcommand{\eps}{\varepsilon}

\newcommand{\bproof}{\begin{proof}}
\newcommand{\eproof}{\end{proof}}

\newlength{\fixboxwidth}
\newcommand{\be}{\begin{equation}}
\newcommand{\ee}{\end{equation}}
\newcommand{\beq}{\begin{eqnarray}}
\newcommand{\beqq}{\begin{eqnarray*}}
\newcommand{\eeq}{\end{eqnarray}}
\newcommand{\eeqq}{\end{eqnarray*}}

\begin{document}
\title{How anisotropic mixed smoothness affects the decay of singular numbers for Sobolev embeddings}

\author{Thomas K\"uhn$^{a, }$
Winfried Sickel$^b$ and Tino Ullrich$^c$
 \\\\
$^a$ Universit\"at Leipzig, Augustusplatz 10, D-04109 Leipzig\\
$^b$ Friedrich-Schiller-Universit\"at Jena, Ernst-Abbe-Platz 2, D-07737 Jena\\
$^c$ TU Chemnitz, Reichenhainer Strasse 39, D-09126 Chemnitz\\\\
}


\date{\today}

\maketitle

\begin{abstract}
We continue the research on the asymptotic and preasymptotic decay of singular numbers for tensor product Hilbert-Sobolev type embeddings in high dimensions with special emphasis on the influence of the underlying dimension $d$. The main focus in this paper lies on tensor products involving univariate Sobolev type spaces with different smoothness. We study the embeddings into $L_2$ and $H^1$. In other words, we investigate the worst-case approximation error measured in $L_2$ and $H^1$ when only $n$ linear measurements of the function are available. Recent progress in the field shows that accurate bounds on the singular numbers are essential for recovery bounds using only function values. The asymptotic bounds in our setting are known for a long time. In this paper we contribute the correct asymptotic constant and explicit bounds in the preasymptotic range for $n$. We complement and improve on several results in the literature. In addition, we refine the error bounds coming from the setting where the smoothness vector is moderately increasing, which has been already studied by Papageorgiou and Wo{\'z}niakowski.

\end{abstract}


\section{Introduction}


In the present paper we aim at approximating $d$-variate functions from tensor product Hilbert-Sobolev type spaces
$$
  H^{s_1} \otimes_2 \cdot \ldots \cdot \otimes_2  H^{s_d}\,,
$$
built upon $L_2(D,\varrho)$. Here, we denote with $H^s:=H^{s,q}$ a univariate Hilbert-Sobolev type space with (fractional smoothness) parameter $s>0$, a fine-index 
$q$, $1\le q \le \infty$, and inner product
\[
	\langle f,g \rangle_{H^{s,q}}:=\sum\limits_{k\in I} c_k(f)\overline{c_k(g)}
	(1+ |k|^{q})^{2s/q}\,,
\]
where $c_k (f)$ denotes the Fourier coefficient of $f\in L_2(D,\varrho)$ with respect to a given orthonormal basis $(e_k)_{k\in I}$ of $L_2(D,\varrho)$  indexed by $I = \N_0$ or $I = \zz$. Note that this setting is rather general since $L_2$ may be any space of square integrable univariate functions on an interval $D$ with respect to a measure $\varrho$. The induced norms $\|\cdot|H^{s,q}\|$ are equivalent for different $q$ and fixed $s$ such that the topology is invariant with respect to $q$. However, in the sequel we will point out the particular role of the index $q$, such that we keep it in the notation of the norm. In fact, its influence on the structure of the unit ball increases if $d$ is getting large. The space
$H^{s_1,q_1}\otimes_2 \cdot \ldots \cdot \otimes_2  H^{s_d,q_d}$
is isometrically isomorphic to the space of all functions from $L_2(D\times...\times D)$ where
\begin{equation}\label{f0}
	\|f|H^{\vec{s},\vec{q}}_{\mix}\|^2:=\sum\limits_{k\in I\times ... \times I} |c_k(f)|^2\prod\limits_{j=1}^d (1+|k_j|^{q_j})^{2s_j/q_j}
\end{equation}
is finite. Now $c_k(f)$ denotes the respective Fourier coefficient for the tensor product system $(e_{k_1} \otimes ... \otimes e_{k_d})_{k \in I \times ...\times I}$ in $L_2(D\times ... \times D)$. The subspace characterized with \eqref{f0} will be denoted by $H^{\vec{s},\vec{q}}_{\text{mix}}$ in the sequel. Particular examples of spaces within this framework are (classical) Sobolev spaces of mixed smoothness $H^{\vec{s},\vec{q}}_{\text{mix}}(\tor^d)$ on the $d$-torus $\tor^d = [0,2\pi]^d$. For these spaces the norm is rather natural since we may rewrite it in 
terms of  derivatives measured in $L_2(\tor^d)$ if $q = 2s$ and $q=\infty$, see Subsection \ref{Sob_torus} below and \cite{KSU2}. Although most of the theorems in this paper are formulated in this periodic framework, the general setting from above allows for dealing also with certain classes of non-periodic functions which are represented in a system different from the tensorized Fourier basis. In fact, using for instance the half period cosine system $(e_k(\cdot))_k = (\cos(\pi k \cdot))_{k \in \N_0}$ the results in this paper also apply to the space of non-periodic functions $H^s_{\mix}([0,1]^d)$ if $s<3/2$ (natural norm in case $s=1)$, see \cite{AdIsNo12} and \cite{WerWoz08}. But also tensor product spaces built upon Legendre, Chebychev or other orthogonal Jacobi polynomials fit into this framework when $D=[-1,1]$ and $d\varrho(x)$ is given by $\nu(x)\,dx$, where $\nu$ is the respective Jacobi weight. The Chebychev case is of particular interest, since the respective spaces are ``close'' to the non-periodic classical Sobolev spaces on $[-1,1]$.

The embedding 
$$
	I_d:H^{\vec{s},\vec{q}}_{\text{mix}}\to L_2\Big(\bigotimes_{i=1}^d D, \varrho^d\Big)
$$
is compact if $\vec{s}>0$ and 
the singular numbers $\sigma_n$ are given as the non-increasing rearrangement of the square root of the eigenvalues of $I_d^{\ast}\circ I_d:H^{\vec{s},\vec{q}}_{\text{mix}}\to H^{\vec{s},\vec{q}}_{\text{mix}}$, which are precisely determined by the reciprocal of the weight appearing in \eqref{f0}. This essentially means that one has access to all the singular numbers. The ``only'' task remaining is to rearrange this multi-indexed sequence in a non-increasing order and to study its decay. 

The singular numbers represent an important tool for the approximation of operators. It is well-known that in case of a compact operator between Hilbert spaces $X$ and $Y$ approximation numbers $a_n(T:X\to Y)$ (defined below) and singular numbers $\sigma_n$ (often called singular values) of the operator $T: X \to Y$ coincide. The approximation numbers of a bounded linear operator $T:X\to Y$ between two Banach spaces are defined as
\be\label{0002}
 \begin{split}
a_n (T:\, X \to Y)&:= \inf\limits_{\rank A<n} \,\sup\limits_{\|x|X\|\leq 1}\|\, Tx -Ax|Y\|\\
&= \inf\limits_{\rank A<n}\|T-A:X \to Y\| , \qquad n \in \N\,.
 \end{split}
\ee
{~}\\
They describe the best approximation of $T$ by finite rank operators.
Note that for compact operators between Hilbert spaces the approximation numbers coincide with all other $s$-numbers, like Kolmogorov or Gelfand numbers, see, e.g., \cite[Section 11.3.]{Pi78}.

Let us emphasize that in many recent works the recovery of functions  represented in a tensorized polynomial basis plays an important role for the treatment of parametric and stochastic elliptic PDEs, see for instance \cite{ChkCoMiNoTe15} or \cite{RaWa12} and the references therein. The mentioned recovery is often in the sampling sense, i.e., from $n$ given function values. It turned out recently, see Krieg, M. Ullrich \cite{KrUl19} and also K\"ammerer, Volkmer, T. Ullrich \cite{KUV}, that accurate bounds on the approximation numbers of an embedding directly yield recovery guarantees for the sampling recovery problem, i.e., new bounds for sampling numbers $g_n$. These represent a counterpart of $a_n$ where the operator $A$ in \eqref{0002} is a sampling operator taking only $n$ function values as information. Let us also mention the recent paper by D\~ung, Thao \cite{DuTh20} in this context.


From a technical viewpoint it will be convenient for us to rearrange the variables according to the smoothness $s_j$ in 
the respective direction. Let $\vec{s} := (s_1, \ldots \, ,s_d)$ be such that
\be\label{ws02}
s_1 = s_2 = \ldots = s_\nu < s_{\nu +1} \le \ldots \le s_d
\ee
for some $\nu$, $1\le \nu <d$.
This can  be interpreted as an ordering with respect to the importance of the variables. Assume for instance, that $f \in H^{\vec{s},\vec{q}}_{\text{mix}}(\tor^d)$ is a rank-$1$ tensor, i.e.,
\[
 f(x)= f_1(x_1)\, \cdot \, \ldots \, \cdot f_d(x_d)\,,
\]
with $x_i \in D,  i=1,...,d$. Then $f_1, \ldots \, f_\nu$ are those univariate  functions which require more effort to be approximated with a certain precision than the others.
In this sense we say that the variables $x_1, \, \ldots \, , x_\nu$ are more important than
$x_{\nu+1}, \, \ldots \, ,x_d$. Note that finding these important variables is a difficult task when it comes to algorithms. A possible ``algorithm'' $A$, which realizes the $n$th approximation number,  is often supposed to ``know'' in which direction the function is ``rough'' and in which direction the function is ``smooth''. This would be sufficiently determined by knowing the parameters of the function class and, in particular, which eigenfunctions are more important than others. Note that our setting represents one way to reduce the number of ``important variables''. There exist various other models. One approach works via introducing weights in the norm on groups of variables. The size of these weights directly determine the influence of certain groups of variables. Let us refer at least to the pioneering work of  Sloan and  Wo{\'z}niakowski \cite{SW}.
More references and detailed information may be found in the monograph \cite{NoWo08} and the more recent papers by
Werschulz and Wo{\'z}niakowski \cite{WW} or D\~ung, Ullrich \cite{DiUl13}. There are some other important problem-related approaches 
which, however,  are not directly related to our investigations here.

 For simplicity we will stick to the classical example of periodic Sobolev spaces with mixed smoothness on 
the $d$-torus $\tor^d$ in order to connect directly to the forerunners \cite{KSU1, KSU2, KMU, Ku, CoKuSi15} and \cite{DK}. 
The mechanism how to transfer the results to different settings is rather clear. It is well-known since the 1960s, 
see e.g. Mityagin \cite{Mi} and Telyakovskii \cite{Tel}, that the asymptotic decay is given by
$$
   c_d n^{-s_1}(\log n)^{(\nu-1)s_1} \leq a_n(I_d:H^{\vec{s},\vec{q}}_{\text{mix}}(\tor^d)\to L_2(\tor^d)) \leq C_d n^{-s_1}(\log n)^{(\nu-1)s_1}\,.
$$
This gives a sharp rate of convergence, where only the number of important variables (in the above sense) plays a role. 
However, since the rate of convergence does not involve any dependence on $d$ one may expect that the dimension 
$d$ (and $q$) shows up in the constants. This is indeed the case. In this paper we give the following statement on 
the ``asymptotic constant'' which extends the results from \cite{KSU2} to the anisotropic mixed smoothness case. It holds
\be\label{ws13_i}
C(d):= \lim_{n\to\infty}  \frac{n^{s_1}\,a_n (I_d: \, H^{\vec{s},\vec{q}}_\mix (\T)\to L_2 (\T))}{(\ln n)^{(\nu-1)s_1}}
= \Big[\frac{2^\nu}{(\nu - 1)!}\, \prod_{j= \nu+1}^d B_j\Big]^{s_1}\, ,
\ee
where
\[
B_j := 1+ 2 \sum_{m=1}^\infty (1+ m^{q_j})^{-\frac{s_j}{s_1\, q_j}}\, , \qquad j=\nu + 1 \, , \ldots \, d\, .
\]
If a sequence $(s_j)_{j=1}^\infty$ with
$$
s_1 =...= s_\nu < s_{\nu+1} \le s_{\nu+2}\le ...
$$
is given, we may consider the $d$-indexed family of embeddings corresponding
to the smoothness vectors $(s_1,...,s_d)$, $d\in\N$. Then, for $d \ge \nu$, the
constants $C(d)$ are strictly increasing in $d$, since all $B_j > 1$. However,
for certain constellations of the parameters, the constants $C(d)$ stay bounded,
for instance if $\vec{q}=\vec{1}$ and
$$
s_j \ge s_1(1+\beta \log_2 j)
$$
for all $j \ge 2$ and some $\beta > 1$\,. 

In any case, such a result does not tell much about the preasymptotic range. The range of small $n$ (say below $2^d$) 
represents the important range for numerical computations. In order to achieve reasonable bounds in the preasymptotic 
range, we use a technique in Section 4 which has its origins in a different context. It is based on an elementary 
counting lemma, see Lemma \ref{clever}. In its simplest form the Lemma deals with upper bounds for the cardinality 
of Zaremba crosses, see Kuo, Sloan, Wo{\'z}niakowski \cite{KSW},  Cools, Kuo, Nuyens \cite{CKN}. Recently Krieg \cite{DK} 
also used a version of this lemma to rearrange tensor power sequences. Our version of this counting technique allows for 
improving on the preasymptotic bounds in \cite{KSU2, Ku, DK} when $\vec{s}$ is a constant vector, i.e. $\nu = d$. Indeed, 
we observe in Theorem \ref{smalldd} below that for $d\geq 3$, $s>0$ and $\vec{q}= (q,q,\ldots \, , q)$ for some $q\geq 1$
\begin{equation}\label{f1}
	a_n (I_d: \, H^{\vec{s},\vec{q}}_\mix (\T)\to L_2 (\T)) \leq \Big(\frac{\tilde{C}(d)}{n}\Big)^{\frac{s}{q(1+\log_2(d-1))}}\,\quad,\quad n\geq 2.
\end{equation}
Note that the constant $\tilde{C}(d)$ ranges in the interval $[2.718, 6.25]$ although it depends on $d$. To be more precise, if $d>5$ then the constant is strictly smaller than $5$, 
which represents a small improvement over \cite{Ku} and \cite{DK}. 
Indeed, the case $q = 2s$ is particularly important since it represents a natural Sobolev norm where 
only the highest derivative is taken into account. It seems that in this case a higher smoothness does 
not increase the exponent in the bound. This observation can be already found in \cite{DK}. Note that, since the range for 
$n$ is not limited, we may easily infer tractability results from this bound (quasi-polynomial tractability, see \cite{KSU2}).


In case of a non-constant smoothness vector $\vec{s}$ the situation is more involved. If the ``first jump'' from $s_\nu$ to $s_{\nu+1}$ is ``small'' then one may use the natural embedding into $H^{\vec{s_1},q}_{\mix}$, where
$\vec{s_1}=(s_1,...,s_1)$, and apply the previously mentioned results. If the first jump is large, say logarithmically in $d$, then the influence of the less important variables $x_{\nu+1},...,x_d$ disappears. Indeed, Theorem \ref{Small1} gives the following bound in case
$$
   \frac{s_{\nu+1}}{s_1} \ge \frac{\log_2 (d-\nu)}{1+ \log_2 (\nu-1)}\, .
$$
It holds
\[
a_n(I_d:H^{\vec{s},\vec{1}}_{\mix}(\T) \to L_2(\T)) \leq \Big(\frac{38.02}{n}\Big)^{\frac{s_1}{1+\log_2 (\nu-1)}}\quad,\quad n\geq 2.
\]
This result indicates that a logarithmic growth condition on the smoothness vector may lead to a polynomial decay of the approximation numbers also in high dimensions $d$. We study this phenomenon in Subsection \ref{numberc}.
In a certain sense this supplements the findings of
Papageorgiou and Wo{\'z}niakowski \cite{PW} by giving precise constants and approximation rates. There the authors prove that the following assertions are equivalent:
\begin{itemize}
 \item There exists  constants $C>0$ and  $p>0$ such that for all $d\in \N$ and all $n \in \N$
 \[
  a_n (I_d:~ H^{s_1}(\tor) \otimes_2 \cdot \ldots \cdot \otimes_2  H^{s_d}(\tor) \to L_2(\T) ) \le C \, n^{-p}\, .
 \]
 \item
There exists constants $C>0$, $p>0$ and  $q>0$ such that for all $d\in \N$ and all $n \in \N$
 \[
  a_n (I_d:~ H^{s_1}(\tor) \otimes_2 \cdot \ldots \cdot \otimes_2  H^{s_d}(\tor) \to L_2(\T)) \le C \, d^q \, n^{-p}\, .
 \]
 \item
 The elements of the sequence $(s_j)_j$ have to increase sufficiently fast, more precisely
 \[
  \limsup_{j\to \infty} \, \frac{\ln j}{s_j} < \infty \, .
 \]
\end{itemize}
In the language of Information Based Complexity
(IBC) the first property is called strong polynomial tractability, the second one polynomial tractability. 
This characterization shows that in case of a moderate growing smoothness vector $\vec{s}$ 
the worst-case errors decay well also in high dimensions. However, it also shows that the problem 
may get more difficult if $s_j$ is not growing properly, which is indicated by the preasymptotic results above. 
In Corollary \ref{smallecd2} we provide the following precise bound when the smoothness vector $\vec{s}$ is 
logarithmically growing. Let $d\geq 2$ and $\vec{q} = \vec{1}$ be the constant $1$-vector. If
\begin{equation}\label{faa}
 s_j \ge (1+ \beta \, \log_2 j) \, s_1\, , \qquad j \in \N\, ,
\end{equation}
for some $\beta >0$ then
\[
a_n(I_d:H^{\vec{s},\vec{q}}_{\mix}(\T) \to L_2(\T)) \leq \Big(\frac{A_\alpha\, e^{C_{\alpha,\beta}}}{n}\Big)^{\frac{s_1}{\alpha}}
\]
for any $\alpha>1/\beta$ and $n\in \N$, see Theorem \ref{smallcdf}. Here $A_\alpha$  and $C_{\alpha,\beta}$ 
are explicit constants in $\alpha $ and $\beta$, see \eqref{Konstante}, Remark \ref{berechnung} and \eqref{konstante4}.

As we will see below, the $d$-dependence of the error decay is determined by  the chosen norm in the source space. To understand  this dependence, we shall work with a family of norms for the univariate Sobolev spaces indexed by $q$. As already mentioned above, these norms are all equivalent on $H^s(\tor)$, the associated norms on the tensor product are,
for fixed $d$, equivalent as well, but the equivalence constants will depend on $d$.
Comparing the results of Section \ref{numberlarge} and Section \ref{numbersmall}
one becomes aware of an enormous difference what concerns the influence of the parameter $\vec{q}$.
Whereas for large $n$ the influence of $\vec{q}$ is only visible in the constants (see \eqref{ws13_i}), the influence is rather strong for small $n$ as \eqref{f1} indicates.


\nocite{GO}
In the final Section \ref{mixediso} we continue the investigations started by Griebel and Knapek \cite{GrKn08} on the decay of the approximation numbers for embeddings into the energy space $H^1(\tor^d)$. These findings have been complemented recently by D\~ung, Ullrich \cite{DiUl13}, Byrenheid et al. \cite{ByDiWiUl14} and two of the authors together with S. Mayer in \cite{KMU}. Based on our new bounds from Section \ref{numbersmall} we are able to essentially improve on these bounds in the literature. This is of interest from at least two points of view. Approximation in the energy norm is of particular importance in connection with
the approximation of solutions of elliptic equations, e.g., the Poisson equation, see \cite{GO}.
Secondly, it is of interest from an inner mathematical point of view.
Whereas $L_2 (\T)$ is a tensor product space, $H^1(\T)$ does not have  such a structure.
So we have a break of the scale when we embed $ H^{s_1}(\tor) \otimes_2 \cdot \ldots \cdot \otimes_2  H^{s_d}(\tor)$ into $H^1(\T)$. 
In \cite{GrKn08} the authors used so-called energy-norm based sparse grids in order to find the index-set for the optimal subspace. 
This index set is essentially determined by the rearrangement of a multi-indexed sequence defined via a quotient of a 
non-tensor product weight and a tensor product weight, see \eqref{f3}. Since this weight is no longer a tensor-product weight we 
avoid rearrangements and rather apply a classical technique mainly used in the field of non-linear approximation. The main results 
read as follows (see Propositions \ref{main1}, \ref{main2}). If $s>1$, $d\in \N$, $d\geq 4$ and $\vec{s} = (s,...,s)$, then
$$
	a_n(I_d:H^{\vec{s},\vec{2}}_{\mix}(\tor^d) \to H^1(\tor^d)) \leq
	\Big(\frac{e^2}{n}\Big)^{\frac{s-1}{2\log_2(d)}}\quad,\quad n\geq 8.
$$
This result is non-trivial in the sense that it improves on the straight-forward bound which \eqref{f1} implies via embedding. 
If $s$ is small compared to $d$, in particular smaller than $\log_2(d)$, we even get the following result, which represents a 
further improvement if $d$ is large:
$$
	a_n(I_d:H^{\vec{s},\vec{2}}_{\mix}(\tor^d) \to H^1(\tor^d))
	\leq \sqrt{d}\Big(\frac{C(d)}{n}\Big)^{\frac{s}{2(1+\log_2(d-1))}}\quad,\quad n\in \N\,,
$$
with
$$
	C(d) = e\,(2.154 + 3/d)\,.
$$

The framework used in this paper essentially goes back to Mityagin \cite[pp. 397, 409]{Mi} in 1962 and 
Telyakovskii \cite[p. 438]{Tel} in 1964 and has been later used by several authors from the former Soviet Union, 
e.g., Galeev \cite{Ga}. See also the book \cite{DuTeUl19} and the references in Section 10.1. The above defined 
spaces experienced a renaissance in the nineties of the last century and are called nowadays anisotropic mixed 
smoothness Sobolev spaces. We refer to \cite{T93b, Tem17, DuTeUl19} and the references therein. 
Within the Information Based Complexity  community  we would like to mention the paper by Papageorgiou and 
Wo{\'z}niakowski \cite{PW}. This paper has initiated further interest, we refer to
\cite{DiGr, DGHR, GHHRW}, \cite{HH}, \cite{IKPW, KPW} and \cite{Sid}.
However, only \cite{PW} is really close to our setting.
Closer to us are the papers by Cobos, K\"uhn, Sickel \cite{CoKuSi15} ($L_\infty$ approximation),
by Krieg \cite{DK} (dominating mixed smoothness, periodic and nonperiodic),
by Wang et al. \cite{CW1,CW2,HW} (anisotropic Sobolev spaces, Sobolev spaces on the sphere) as well as the papers by Mieth \cite{Mie}
and Novak \cite{No19} (approximation on general domains).

The paper is organized as follows. In Section \ref{spaces} we collect some preliminaries like the definition of the spaces
under consideration and some basic properties of approximation and singular numbers.
The next Section \ref{numberlarge}
will be devoted to the study of the asymptotic constants in case that $\vec{s}$ is not a constant vector.
In Section  \ref{numbersmall} we derive estimates for  the approximation numbers $a_n$ in the  preasymptotic range.
Finally, in Section \ref{mixediso} we give new bounds for embeddings of the spaces $H^{\vec{s},\vec{2}}_\mix (\T)$
into $H^1(\T)$.

{\bf Notation.} As usual, $\N$ denotes the natural numbers, $\N_0$ the non-negative integers,
$\zz$ the integers and
$\re$ the real numbers. By $\tor$ we denote the torus, represented by the interval $[0,2\pi]$, where
the end points of the interval are identified.
For a real number $a$ we put $a_+ := \max\{a,0\}$ and denote by $\lfloor a \rfloor$ the greatest integer not larger than $a$.
The letter $d$ is always reserved for the dimension in $\Z$, $\R$, $\N^d$, and $\T$.
For $0<p\leq \infty$ and $x\in \R$ we denote $|x|_p = (\sum_{i=1}^d |x_i|^p)^{1/p}$ with the
usual modification for $p=\infty$. If $\alpha\in \N_0^d$ and $x \in \C^d$ we use
$x^\alpha:=\prod_{i=1}^d x_i^{\alpha_i}$ with the convention $0^0:=1$. The symbol $\# \Omega$ stands for the cardinality of the set $\Omega$.
If $X$ and $Y$ are two Banach spaces, the norm
of an element $x$ in $X$ will be denoted by $\|x|X\|$ and the norm of an operator
$A:X \to Y$ by $\|A:X\to Y\|$. The symbol $X \hookrightarrow Y$ indicates that there is a
continuous embedding from $X$ into $Y$. The equivalence $a_n\sim b_n$ means that there are constants $0<c_1\le c_2<\infty$ such that
$c_1a_n\le b_n\le c_2a_n$ for all $n\in\N$.
If $\vec{s} = (s_1, \ldots \, , s_d)$ and $\vec{q} = (q_1, \ldots \, , q_d)$ are given, then
\[
 \frac{\vec{s}}{\vec{q}} := \Big(\frac{s_1}{q_1}, \, \ldots \, , \frac{{s_d}}{{q_d}}\Big)\, .
\]
Furthermore, $\vec{s} \ge  {\vec{q}}$ means $s_j \ge q_j$ for all $j$.


\section{Preliminaries}\label{spaces}



\subsection{Sobolev spaces with anisotropic mixed smoothness}
\label{Sob_torus}


As mentioned in the introduction, our  setting is rather general and not restricted to periodic functions. 
The essential ingredient is the weight function appearing in \eqref{f0}. However, for simplicity we will state all 
results for function spaces on the $d$-torus $\T$,
which is represented in the Euclidean space $\R$ by the cube
$\tor^d = [0,2\pi]^d$, where opposite faces are identified.
In particular, for functions $f$ on $\tor$, we have $f(x) = f(y)$ whenever $x-y = 2\pi k$ for some $k\in \zz$. These
functions can be viewed as $2\pi$-periodic in each component.

The space $L_2(\T)$ consists of all (equivalence classes of) measurable functions $f$ on $\T$ such that the norm
$$
    \|f|L_2(\T)\|:=\Big(\int_{\T} |f(x)|^2\,dx\Big)^{1/2}
$$
is finite. The entire information of a function $f\in L_2(\T)$ is encoded
in the sequence $(c_k(f))_k$ of its Fourier coefficients, given by
\[
c_k (f):= \frac{1}{(2\pi)^{d/2}} \, \int_{\T} \, f(x)\, e^{-ik\cdot x}\, dx\, , \qquad k
\in
\Z\, .
\]
Indeed, we have Parseval's identity
\be\label{Pars}
    \|f|L_2(\T)\|^2 = \sum\limits_{k\in \Z} |c_k(f)|^2
\ee
as well as
\[
f(x) = \frac{1}{(2\pi)^{d/2}}\, \sum_{k \in \Z} \, c_k (f)\, e^{ik\cdot x}
\]
with convergence in $L_2(\T)$.

The (anisotropic) Sobolev space $H^{\vec{m}}_\mix(\T)$ of  smoothness $\vec{m}\in \N^d$ is the collection
of all $f\in L_2 (\T)$ such that all
distributional partial derivatives $D^\alpha  f$ of order $\alpha = (\alpha_1,...,\alpha_d)$ with
$\alpha_j\le m_j$, $j=1,...,d$, belong to $L_2 (\T)$.
It is usually normed by
\begin{equation}\label{t-1}
\| \, f \, |H^{\vec{m}}_{\mix} (\T)\| := \Big(
\sum_{{\alpha_j \le  m_j \atop j=1, \ldots \, , d}}
\| \, D^\alpha  f \, |L_2 (\T)\|^2\Big)^{1/2} \, .
\end{equation}
One can rewrite this definition in terms of Fourier coefficients.
Taking $c_k(D^{\alpha}f) = (ik)^{\alpha}c_k(f)$ into account, Parseval's identity \eqref{Pars}
implies (note that we put $0^0 := 1$)
\beqq
\| \, f \, |H^{\vec{m}}_{\mix} (\T)\|^2 & = &
\sum_{\substack{\alpha_j \leq m_j \\ j=1,..,d}}
\Big\| \, \frac{1}{(2\pi)^{d/2}}\sum_{k \in \Z} \, c_k (f)\, (ik)^\alpha  e^{ik\cdot x} \,
\Big|L_2 (\T)\Big\|^2
\\
\nonumber
& = & \sum\limits_{k\in \Z} |c_k(f)|^2 \, \Big(\sum_{\substack{0\leq \alpha_j\leq m_j\\
j=1, \ldots ,d}}
\prod\limits_{j=1}^d |k_j|^{2\alpha_j}\Big)\\
\nonumber
&=& \sum\limits_{k\in \Z} |c_k(f)|^2 \, \Big(\prod\limits_{j=1}^d\sum\limits_{\alpha_j = 0}^{m_j}
|k_j|^{2\alpha_j}\Big) =  \sum\limits_{k\in \Z} |c_k(f)|^2 \,\prod\limits_{j=1}^d \omega_{m_j}(k_j)^2\,,
\nonumber
\eeqq
where
\[
  \omega_m(\ell)^2 = \sum\limits_{n=0}^m |\ell|^{2n}\,.
\]
Due to our convention $0^0=1$ we have $\omega_m(0)=1$.
Defining
\[
w_{\vec{m}}(k) := \prod_{j=1}^d \omega_{m_j}(k_j)\qquad\text{for } k = (k_1,...,k_d) \in \Z\,,
\]
we obtain
\be\label{neu2}
   \| \, f \, |H^{\vec{m}}_{\mix} (\T)\| = \Big[\sum\limits_{k\in \Z} |c_k(f)|^2 \, w_{\vec{m}}(k)^2\Big]^{1/2}\,.
\ee
We could also have started with the equivalent norm
\be\label{norm3}
\| \, f \, |H^{\vec{m}}_\mix (\T)\|^* := \Big(\sum_{\substack{\alpha_j \in \{0,m_j\} \\ j=1, \ldots \, d}} \| \, D^{\alpha}f\, |L_2
(\T)\|^2\Big)^{1/2} \, .
\ee
Similarly as above, a reformulation of (\ref{norm3}) in terms of Fourier coefficients yields
\be\label{norm4}
\| \, f \, |H^{\vec{m}}_\mix (\T)\|^*  =
 \,
 \Big[\sum_{k \in \Z} \, |c_k (f)|^2  \prod\limits_{j=1}^d (1+|k_j|^{2m_j}) \Big]^{1/2}\,.
\ee

Inspired by  \eqref{norm4}  we define Sobolev spaces of dominating
mixed anisotropic smoothness of fractional order $\vec{s}$ as follows.

\begin{defi}\label{varnorms}
Let $\vec{s} = (s_1, \ldots \, , s_d)$, $\min_j \, s_j >0$, and let
$\vec{q}= (q_1, \, \ldots \, ,q_d)$ such that $0 <  q_j \le  \infty$ for all $j$.
The periodic  dominating mixed anisotropic Sobolev space $H^{\vec{s},\vec{q}}_{\mix}(\T)$
is the collection of all $f\in L_2(\T)$
such that
\[
   \| \, f \, |H^{\vec{s}, \vec{q}}_{\mix} (\T)\|  :=
\, \Big[
 \sum_{k \in \Z} \, |c_k (f)|^2 \prod\limits_{j=1}^d\big(1+|k_j|^{q_j}\big)^{2s_j/q_j} \Big]^{1/2} < \infty\, , 
\]
where $\big(1+|k_j|^{q_j}\big)^{2s_j/q_j}$
has to be replaced by $\max \big(1, |k_j|\big)^{2s_j}$ if $q_j = \infty$.
\end{defi}

\begin{rem}
 \rm
(i) Obviously we have
$H^{\vec{m},\vec{q}}_{\mix} (\T) = H^{\vec{m}}_{\mix} (\T)$,  $\vec{m}\in \N^d$,
in the sense of equivalent norms. In the fractional case it follows
$H^{\vec{s},\vec{q}}_{\mix} (\T) = H^{\vec{s},\vec{2}}_{\mix} (\T)$
in the sense of equivalent norms.
Here $\vec{2}$ refers to the sequence $(2, \ldots \, ,2)$.
\\
(ii) The sequence of parameters $\vec{q}$ modifies the norm in a controlled way. Here
in our paper it will be used to demonstrate how much certain results depend on the chosen norm.
\\
(iii) If $s = s_1 = s_2 = \ldots =s_d $, then we simply write $H^{s}_{\mix} (\T)$ and call the space Sobolev space  
of dominating mixed smoothness of fractional order $s$.
\\
(iv) The spaces $H^{\vec{s}}_{\mix} (\T)$ have played a significant role in the Russian approximation theory literature, 
see, e.g., the papers by
Mityagin \cite{Mi}, Telyakovskij \cite{Tel}, Nikol'skaya \cite{Nia}, Galeev \cite{Ga} or the
monographs of Temlyakov \cite{T93b} and \cite{Tem17}.
\\
(v) Most important for us will be the cases $\vec{q}=(1, 1,  \ldots \, ,1) = \vec{1}$, $\vec{q}=(2, 2,  \ldots \,, 2) = \vec{2}$
and
$\vec{q}=(\infty, \infty,  \ldots \, ,\infty) = \vec{\infty}$. In the latter case
the norm reads as
\[
   \| \, f \, |H^{\vec{s}, \vec{\infty}}_{\mix} (\T)\|  :=
\, \Big[
 \sum_{k \in \Z} \, |c_k (f)|^2 \prod\limits_{j=1}^d\big(\max (1, |k_j|)\big)^{2s_j} \Big]^{1/2} < \infty \, .
\]
\end{rem}

Clearly, there is a monotonicity of the norms with respect to $\vec{q}$, i.e.,
\begin{equation}\label{ws-unendlich}
\|\, f\, |H^{\vec{s}, \vec{\infty}}_{\mix}(\T)\|\le
\|\, f\, |H^{\vec{s},\vec{q}}_{\mix}(\T)\|
\end{equation}
for all $f\in H^{\vec{s}}_\mix(\T)$.

Later on we shall need the following observation.
Let $\vec{s}$ and  $\vec{q}$ be given.
Then we define
\[
\vec{s}/\vec{q} =
 \frac{\vec{s}}{\vec{q}}:= \Big(\frac{s_1}{q_1}, \ldots \, , \frac{s_d}{q_d}\Big)\, .
\]

\begin{lem}\label{wichtig} Let $\vec{s}= (s_1, \ldots \, s_d)$, $\min_j s_j >0$.
Then for all $\vec{q}$, $1 \le  q_j< \infty $, $j=1, \ldots \, , d$,  and all
$f \in H^{\vec{s},\vec{q}}_{\mix}(\T)$ the following inequality holds
\[
\|\, f\, |H^{\vec{s}/\vec{q}, \vec{1}}_{\mix}(\T)\|\le
\|\, f\, |H^{\vec{s},\vec{q}}_{\mix}(\T)\| \, .
\]
\end{lem}

\begin{proof}
 Since $1 \le q_j < \infty$  we have for all $k \in \Z$ and all $j$
 \[
  (1+|k_j|)\le (1+|k_j|^{q_j}) \, .
 \]
 Taking this to the power $2 s_j/q_j$ and switching to the product we obtain
 \[
 \prod_{j=1}^d (1+|k_j|)^{2 s_j/q_j} \le \prod_{j=1}^d (1+|k_j|^{q_j})^{2 s_j/q_j}
\]
Hence
\[
 \sum_{k \in \Z} \, |c_k (f)|^2 \prod\limits_{j=1}^d\big(1+|k_j|\big)^{2s_j/q_j}
 \le
 \sum_{k \in \Z} \, |c_k (f)|^2 \prod\limits_{j=1}^d\big(1+|k_j|^{q_j}\big)^{2s_j/q_j}
 \]
 which proves the claim.
\end{proof}

\begin{rem}\rm
Observe that this simple argument used in the proof does not extend to vectors $\vec{q}$
containing at least one component in $(0,1)$.
\end{rem}

Later on we shall also need the following elementary fact.

\begin{lem}\label{norm_two} Let $\vec{m} \in \N^d$.
Then for all $f \in H^{\vec{m}}_{\mix}(\T)$ the following chain of inequalities holds.
\[
\|\, f\, |H^{\vec{m}, \vec{\infty}}_{\mix}(\T)\|\le
\|\, f\, |H^{\vec{m}}_{\mix}(\T)\|^* \le
\|\, f\, |H^{\vec{m}}_{\mix}(\T)\| \le  \|\, f \, |H^{\vec{m},\vec{2}}_{\mix}(\T)\|\, .
\]
\end{lem}

In the Introduction we were dealing with tensor products of univariate Sobolev spaces.
The connection to the spaces $H^{\vec{s},\vec{q}}_{\mix}(\T)$ will become clear with the next lemma.
Observe that Definition \ref{varnorms} makes sense also for $d=1$.
For two Hilbert spaces $H_1, H_2$ the symbol
$H_1 \otimes_2 H_2 $ denotes their tensor product, see, e.g., \cite[3.4]{We}
for the basics.
The symbol $H_1 \otimes_2 \ldots \otimes_2 H_d $ has to be interpreted as iterated tensor product, i.e.,
\[
H_1 \otimes_2 H_2 \otimes_2 H_3 := (H_1 \otimes_2 H_2) \otimes_2 H_3
\]
and so on.

\begin{lem}\label{tensor}
 Let $\vec{s} = (s_1, \ldots \, , s_d)$, $s_j \in \re$ for all $j$, and let
$\vec{q}= (q_1, \, \ldots \, ,q_d)$ such that $0 <  q_j \le  \infty$ for all $j$.
Then periodic anisotropic mixed Sobolev space $H^{\vec{s},\vec{q}}_{\mix}(\T)$
coincides with the tensor product of the univariate Sobolev spaces
$H^{{s_1},{q_1}}(\T)$, $\ldots$, $H^{{s_d},{q_d}}(\T)$.
More exactly, we have
\[
H^{\vec{s},\vec{q}}_{\mix}(\T)=
  H^{s_1,q_1}(\tor) \otimes_2 \cdot \ldots \cdot \otimes_2  H^{s_d,q_d}(\tor) \, ,
\]
and
\[  \| \, \cdot \, |H^{\vec{s}, \vec{q}}_{\mix} (\T)\| = \|\, \cdot\, |
 H^{s_1,q_1}(\tor) \otimes_2 \cdot \ldots \cdot \otimes_2  H^{s_d,q_d}(\tor)\|
 \, .
\]
\end{lem}

\begin{proof} For convenience of the reader we will give a proof.
We will use the following fact. If $(e^1_j)_{j=0}^\infty$ is an orthonormal basis 
in the  Hilbert space $H_1$ and 
if $(e^2_j)_{j=0}^\infty$ is an orthonormal basis 
in the Hilbert space $H_2$, then 
 $(e^1_j \otimes e^2_{\ell})_{j,\ell=0}^\infty$ is an orthonormal basis 
in the tensor product Hilbert space $H_1 \otimes_2 H_2$, see, e.g., 
\cite[Satz 3.12 on pages 52/53]{We}.
Clearly, 
\[
 \frac{e^{ikx}}{\sqrt{2\pi}\, \big(1+|k|^{q}\big)^{s/q}}\, , \qquad 
 x\in \re, \quad  k\in \zz\, , 
\]
is an orthonormal basis in $H^{s,q}(\tor)$ and
\[
e_{\vec{k}}(x):=  \prod_{j=1}^d \frac{e^{ik_jx_j}}{\sqrt{2\pi}\, \big(1+|k_j|^{q_j}\big)^{s_j/q_j}}\, , \qquad 
 x\in \R, \quad  k\in \zz^d\, , 
\]
is an orthonormal basis in $H^{\vec{s},\vec{q}}_{\mix}(\T)$.
Let us turn to $d=2$ for a moment.
By means of the quoted result the functions 
$(e_{\vec{k}})_{\vec{k}\in \zz^2}$ form an orthonormal basis for 
$H^{s,q}(\tor) \otimes_2 H^{s,q}(\tor)$ as well. From this fact we conclude
\[
 \Big\| \, \sum_{\vec{k}\in I} a_{\vec{k}}\, e_{\vec{k}}\, \Big|H^{s,q}(\tor) \otimes_2 H^{s,q}(\tor)\Big\| = 
\Big(\sum_{\vec{k} \in I} |a_{\vec{k}}|^2\Big)^{1/2} = 
\Big\| \, \sum_{\vec{k}\in I} a_{\vec{k}}\, e_{\vec{k}}\, \Big|
H^{\vec{s},\vec{q}}_{\mix}(\T)\Big\|
 \]
for any set $I \subset \zz^2$ of finite cardinality and any sequence
$(a_{\vec{k}})_{\vec{k}}$ of complex numbers.
Those functions $\sum_{\vec{k}\in I} a_{\vec{k}}\, e_{\vec{k}}$
are dense in $H^{s,q}(\tor) \otimes_2 H^{s,q}(\tor)$ by definition 
and in 
$H^{\vec{s},\vec{q}}_{\mix}(\T)$ by a short calculation.
Hence, the spaces and the norms coincide. The general case $d\ge 2$
follows by induction.
\end{proof}

Also the Sobolev space $H^{\vec{m}}_{\mix}(\T)$ can be interpreted as a tensor product of univariate Sobolev spaces.

\begin{lem}\label{tensor2}
 Let $\vec{m} = (m_1, \ldots \, , m_d)$, $m_j \in \N_0$ for all $j$.
Then the periodic anisotropic mixed Sobolev space $H^{\vec{m}}_{\mix}(\T)$
coincides with the tensor product of the univariate Sobolev spaces
$H^{{s_1}}(\T)$, $\ldots$, $H^{{s_d}}(\T)$.
More exactly, we have
\[
H^{\vec{m}}_{\mix}(\T)=
  H^{m_1}(\tor) \otimes_2 \cdot \ldots \cdot \otimes_2  H^{m_d}(\tor) \, ,
\]
and
\[  \| \, \cdot \, |H^{\vec{m}}_{\mix} (\T)\| = \|\, \cdot\, |
 H^{m_1}(\tor) \otimes_2 \cdot \ldots \cdot \otimes_2  H^{m_d}(\tor)\|
 \, .
\]
\end{lem}

\begin{proof}
 The proof follows by the same type of arguments as the proof of Lemma \ref{tensor}.
\end{proof}


\subsection{Singular numbers of diagonal operators}
\label{sing_diag}


If $\tau = (\tau_n)_{n=1}^{\infty}$ is a sequence of real numbers with $\tau_1 \geq \tau_2 \geq ... \geq 0$\,,
we define the diagonal operator
$  D_{\tau}:\ell_2 \to \ell_2$ by $D_{\tau}(\xi) = (\tau_n \xi_n)_{n=1}^{\infty}$.
Recall the definition of the approximation numbers \eqref{0002} already given in the Introduction.
The following fact concerning approximation numbers of diagonal operators is well-known,
see e.g. Pietsch \cite[Theorem 11.3.2.]{Pi78}, K\"onig \cite[Section 1.b]{Ko}, Pinkus \cite[Theorem IV.2.2]{Pin},
or  Novak and Wo{\'z}niakowski \cite[Corollary 4.12]{NoWo08}. Comments on the history may be found in
Pietsch \cite[6.2.1.3]{Pi07}.

\begin{lem}\label{sing}
Let $\tau$ and $D_\tau$ be as above.
Then
\[
a_n (D_\tau:\, \ell_2 \to \ell_2) = \tau_n\, , \qquad n \in \N\, .
\]
\end{lem}
\noindent Here the index set of $\ell_2$ is $\N$. We need a modification for arbitrary
countable index sets $J$. Then the space $\ell_2(J)$ is the collection of all
$\xi = (\xi_j)_{j\in J}$ such that the norm
$$
    \|\xi|\ell_2(J)\|:=\Big(\sum\limits_{j\in J}|\xi_j|^2\Big)^{1/2}
$$
is finite. Let $w = (w_j)_{j\in J}$ with $w_j>0$ for all $j\in J$, and assume that for every $\delta>0$ there are
only finitely many $j\in J$ with $w_j\geq \delta$\,. Then the non-increasing rearrangement
$(\tau_n)_{n\in \N}$ of $(w_j)_{j\in J}$ exists, and $\lim_{n\to\infty} \tau_n = 0$.
Defining $D_w:\ell_2(J) \to \ell_2(J)$ by $D_w(\xi) = (w_j\xi_j)_{j\in J}$ for $\xi \in \ell_2(J)$, Lemma\
\ref{sing} gives
$$
  a_n(D_w:\ell_2(J)\to \ell_2(J)) = \tau_n\,.
$$
The preceding identity is scalable in the following sense.

\begin{lem}\label{scale}
Let $J$ be a countable index set, let $w=(w_j)_{j\in J}$ and $(\tau_n)_{n\in \N}$ be as above. Then,
setting $w^s=(w_j^s)_{j\in J}$, one has for any $s>0$
$$
  a_n(D_{w^s}:\ell_2(J)\to \ell_2(J)) = a_n(D_{w}:\ell_2(J)\to \ell_2(J))^s = \tau^s_n\,.
$$
\end{lem}

Now we can reduce our problem on embedding operators in function spaces to the considerably simpler context of diagonal operators in sequence spaces,
where the index set is $J=\Z$\,. We consider the operators
$$
  A_{\vec{s},\vec{q}}:~H^{\vec{s},\vec{q}}_{\mix}(\T) \to \ell_2(\Z)\qquad
\mbox{and}\qquad B_{\vec{s},\vec{q}}:~\ell_2(\Z) \to H^{\vec{s},\vec{q}}_{\mix}(\T)
$$
defined by
$$
   A_{\vec{s},\vec{q}}f := (u_{\vec{s},\vec{q}}(k) \, c_k(f))_{k\in \Z}\qquad
\mbox{and}\qquad B_{\vec{s},\vec{q}} \xi := (2\pi)^{-d/2}\sum\limits_{k\in \Z}
   \frac{\xi_k}{u_{\vec{s},\vec{q}}(k)} e^{ik\cdot x}\,,
$$
where the weights  $u_{\vec{s},\vec{q}}(k)$ are given by
\[
u_{\vec{s},\vec{q}}(k) := \prod_{j=1}^d\big(1+|k_j|^{q_j}\big)^{s_j/q_j}
\]
(standard modification if $q_j = \infty$).

Note the semi-group property of these weights, i.e., $u_{\vec{s},\vec{q}}(k)\cdot u_{\vec{t},\vec{q}}(k) = u_{\vec{s} + \vec{t},\vec{q}}(k)$.
Furthermore, we put for $k\in \Z$
\[
w(k) := \frac{u_{\vec{s},\vec{q}}(k)}{u_{\vec{t},\vec{q}}(k)}
\]
and make use of the associated diagonal operator $D_w$.
Then the following commutative diagram illustrates the situation
quite well in case $\vec{t}> \vec{s} \geq \vec{0}$:

\begin{center}
\tikzset{node distance=3cm, auto}

\begin{tikzpicture}
  \node (H) {$H^{\vec{t},\vec{q}}_{\mix}(\T)$};
  \node (L) [right of =H] {$H^{\vec{s},\vec{q}}_{\mix}(\T)$};
  \node (ell) [below of=H] {$\ell_2(\Z)$};
  \node (ell2) [right of=ell] {$\ell_2(\Z)$};
  \draw[->] (H) to node {$I_d$} (L);
  \draw[->] (H) to node {$A_{\vec{t},\vec{q}}$} (ell);
  \draw[->] (ell) to node {$D_w$} (ell2);
  \draw[->] (ell2) to node {$B_{\vec{s},\vec{q}}$} (L);
\end{tikzpicture}
\end{center}

By the definition of the norm $\|\cdot|H^{\vec{s},\vec{q}}_{\mix}(\T)\|$ it is clear
that $A_{\vec{s},\vec{q}}$ and $B_{\vec{s},\vec{q}}$ are isometries and $B_{\vec{s},\vec{q}} = (A_{\vec{s},\vec{q}})^{-1}$. For the embedding
$I_d:H^{\vec{t},\vec{q}}_{\mix}(\T) \to H^{\vec{s},\vec{q}}_{\mix}(\T)$ if $\vec{t}> \vec{s} \geq \vec{0}$ we obtain the factorization
\be\label{eq071}
    I_d =  B_{\vec{s},\vec{q}} \circ D_w \circ A_{\vec{t},\vec{q}}\,.
\ee
The multiplicativity of the approximation numbers applied to \eqref{eq071} implies
$$
  a_n(I_d) \leq \|A_{\vec{t},\vec{q}}\|\, \, a_n(D_w)\, \, \|B_{\vec{s},\vec{q}}\| = a_n(D_w) = \tau_n\,,
$$
where $(\tau_n)_{n=1}^{\infty}$ is the non-increasing rearrangement of
$(w(k))_{k\in \Z}$\,. The reverse inequality can be shown analogously. This gives the important identity
\be\label{eq072}
    a_n(I_d) = a_n(D_w) = \tau_n\,.
\ee
Let $\lambda>0$.
Due to the semi-group property mentioned above and Lemma \ref{scale}
we have in particular the nice property
\be\label{eq073}
    a_n(I_d:H^{\lambda \vec{s},\vec{q}}_{\mix}(\T)\to L_2 (\T)) = a_n(I_d:H^{\vec{s},\vec{q}}_{\mix}(\T)\to L_2(\T))^\lambda \, .
 \ee

Finally let us mention one more elementary fact.

\begin{lem}\label{back}
 Let $X,Y$ and $Z$ be Banach spaces.
 Suppose $X \hookrightarrow Y \hookrightarrow Z$. Denote by $I^1,I^2,I^3$ the identities
 mapping $X $ into $Y$,  $Y$ into $Z$ and $X$ into $Z$.
 If  $\| \,I^1 \, |X \to Y\| \le 1$ then
 \[
  a_n (I^3: ~ X \to Z)\le  a_n (I^2: ~ Y \to Z)
 \]
holds for all $n$.
\end{lem}

\begin{proof}
 It is enough to have a look at the commutative diagram

\centerline{\begin{minipage}[u]{7cm}
\tikzset{node distance=5cm, auto}
\begin{tikzpicture}
  \node (H) {$X$};
  \node (L) [right of =H] {$Y$};
  \node (L2) [right of =H, below of =H, node distance = 2.5cm ] {$Z$};
    \draw[->] (H) to node {$I^1$} (L);
  \draw[->] (H) to node [swap] {$I^3$} (L2);
  \draw[->] (L) to node {$I^2$} (L2);
  \end{tikzpicture}
\end{minipage}}
\noindent
Because of $I^3 = I^1 \circ I^2$ the multiplicativity of the approximation numbers
yields the claim.
\end{proof}


\subsection{Approximation numbers of Sobolev embeddings}


In this subsection we will make a few more observations on the approximation numbers of the embedding
$I_d: \, H^{\vec{t},\vec{q}}_{\mix} (\T) \to L_2 (\T)$\,.
Recall
\[
 u_{\vec{t},\vec{q}}(k):= \prod\limits_{j=1}^d(1+ |k_j|^{q_j})^{t_j/q_j}\, ,
 \qquad k \in \Z \, ,
\]
(modification if $q_j = \infty$).
Due to Lemma \ref{sing} and \eqref{eq072} we have
\[
a_{n} (I_d: \, H^{\vec{t},\vec{q}}_{\mix} (\T) \to L_2 (\T)) = \tau_n\, , \qquad n \in \N\, ,
\]
where $(\tau_n)_{n\in\N}$ denotes the non-increasing rearrangement
of $(1/u_{\vec{t},\vec{q}}(k))_{k \in \Z}$.
We need a notation for the associated sequence of real numbers without repetitions.
Let
$(\vartheta_n)_{n=1}^\infty$ be the sequence of positive  real numbers such that
\[
\vartheta_1 := 1 < \vartheta_2 := 2^{\min_{j=1, \ldots \, d} s_j/q_j} < \vartheta_3 < \ldots < \vartheta_n < \ldots
\]
and
\be\label{ws23}
\{\vartheta_n: n\in\N\} =
\{u_{\vec{t},\vec{q}}(k): k\in\Z\}\,.
\ee
Define
\be\label{cardinal}
C(r,\vec{t}, \vec{q})
 := \# \, \Big\{ k \in \Z: \quad \prod_{j=1}^d\, (1+|k_j|^{q_j})^{t_j/q_j}  \le r\Big\}\, , \qquad
r \geq 1\, ,
\ee
(modification if $q_j = \infty$).
Then the function $f(r):= C(r,\vec{t}, \vec{q})$ is a piecewise constant function, the jumps are located in the points $\vartheta_n$.
These observations imply the following.

\begin{lem}\label{an1}
Let $\vec{t}>0$ and $\vec{q}$ be given.
Then, with  $n_m:= C(\vartheta_m,\vec{t}, \vec{q})$, $m \in \N$, we have
$$
a_{n_m} (I_d: \, H^{\vec{t},\vec{q}}_{\mix} (\T) \to L_2 (\T)) = 1/\vartheta_m  \, , \qquad m \in \N\, ,
$$
and
\[
 a_{n_m + 1} (I_d: \, H^{\vec{t},\vec{q}}_{\mix} (\T) \to L_2 (\T))
= \ldots =  a_{n_{m + 1}} (I_d: \, H^{\vec{t},\vec{q}}_{\mix} (\T) \to L_2 (\T))  = 1/\vartheta_{m+1}\, ,
\]
$m \in \N$.
\end{lem}

\begin{rem}\label{opti}
\rm
(i)
Of course, without precise information on the behavior of the quantities
$C(r,\vec{t}, \vec{q})$, Lemma \ref{an1} is not very useful
for practical purposes. But it provides, at least in
principle, complete knowledge on the sequence of approximation numbers \\
$a_n (I_d: \, H^{\vec{t},\vec{q}}_{\mix} (\T) \to L_2 (\T))$.
\\
(ii)
For any $n\in \N$, we can easily construct optimal algorithms $S_n$ of rank less than $n$.
We choose a set $\Lambda_n \subset \Z$ with the following properties:
the cardinality of $\Lambda_n$ equals $n-1$ and if $\ell \not\in \Lambda_n$
then
\[
\sup_{k \in \Lambda_n}\,  u_{\vec{t},\vec{q}}(k)
\le u_{\vec{t},\vec{q}}(\ell) \, .
\]
follows. With other words, we select $n-1$ vectors $k$ such that the associated
values $u_{\vec{s},\vec{q}}(k)$ are the smallest.
Then we define
\be\label{opti1}
S_n f (x):= \frac{1}{(2\pi)^{d/2}} \, \sum_{k \in \Lambda_n} c_k (f) \,
e^{ik\cdot x}\, , \qquad x \in \T\, .
\ee
Let $C(\vartheta_m,\vec{s},\vec{q}) < n \le
C(\vartheta_{m+1},\vec{s},\vec{q})$. By this construction we get
$$
    \sup\limits_{\|f|H^{\vec{t},\vec{q}}_{\mix}(\T)\| \leq 1} \, \|f-S_nf|L_2(\T)\|
= \frac{1}{\vartheta_{m+1}} = a_{n}(I_d: \, H^{\vec{t},\vec{q}}_{\mix} (\T) \to L_2 (\T))\,.
$$
\end{rem}


\section{Asymptotic behavior and constants}
\label{numberlarge}


Taking our convention \eqref{ws02} into account,
Mityagin \cite{Mi} was the first who showed the two-sided estimate
\begin{equation}\label{ws-06}
c_{\vec{s},\vec{q}}(d,\nu)  \, n^{-s_1}(\ln n)^{(\nu-1)s_1} \le
a_{n} (I_d: \, H^{\vec{s},\vec{q}}_\mix (\T) \to L_2 (\T)) \le
 C_{\vec{s},\vec{q}}(d,\nu)  \, n^{-s_1}(\ln n)^{(\nu-1)s_1}\,,
\end{equation}
for $n\in\N$, $n>1$. Here the constants $c_{\vec{s},\vec{q}}(d,\nu)$ and
$C_{\vec{s},\vec{q}}(d,\nu)$, depending only on $d,\nu, \vec{s}$ and $\vec{q}$,
were not explicitly determined.
Our main focus is  to clarify, for arbitrary but fixed $d$, $\nu$, $\vec{s}$ and $\vec{q}$, the
dependence of these constants on $d$ and $\nu$.
In fact, it is necessary to fix the norms, i.e., the $\vec{q}$ on the spaces
$H^{\vec{s}}_\mix (\T)$ in advance, since the
constants $c_{\vec{s},\vec{q}}(d,\nu)$ and $C_{\vec{s},\vec{q}}(d,\nu)$ in \eqref{ws-06} depend on the size of the respective unit balls.

We recall a result basically proved in  \cite{KSU2}.

\begin{prop}\label{lim2}
Let $d\in \N$. Let $\vec{s}$ be given by $s_1 = \ldots = s_d>0$ and $\vec{q}>0$ arbitrary.
Then
\be\label{nochwas}
 \lim_{n\to\infty}  \frac{n^{s_1}\,a_n (I_d: \, H^{\vec{s},\vec{q}}_\mix (\T)\to L_2 (\T))}{(\ln n)^{(d-1)s_1}}
= \Big[\frac{2^d}{(d - 1)!}\Big]^{s_1}\, .
\ee
\end{prop}

\begin{proof}
{\em Step 1.} Let $\vec{q}$ be a constant vector generated by some $q\in (0,\infty]$. Let
$s_1 = \ldots = s_d =1$.
We shall  employ \cite[Thm.\ 4.3]{KSU2} together with \cite[Lem.\ 4.14]{KSU2}.
For convenience of the reader we first recall this lemma.
For $\ell\in\mathbb{Z}$, $0<\eps\le 1$ and $d\in\N$ let
$$
y_\ell:=\frac{1}{1+|\ell|}\quad,\quad\mathcal{Y}_d(\eps) :=\left\{k\in\Z:\: y_{k_1}\cdots y_{k_d}\ge \eps \right\}\quad,\quad
Y_d(\eps):= \# \mathcal{Y}_d(\eps)\,.
$$

\begin{lem}\label{combi} Let $(z_\ell)_{\ell\in\zz}$ be a sequence indexed by $\zz$ such that
$$
0<z_{\ell}\le z_0=1 \quad\text{for all } \ell\neq 0\qquad\text{and }\quad\lim_{|\ell|\to\infty}\frac{y_\ell}{z_\ell}=1\,.
$$
Similarly as for $(y_\ell)_{\ell\in\zz}$ we define
$\mathcal{Z}_d(\eps)$ and $Z_d(\eps)$ associated to $(z_\ell)_{\ell\in\zz}$\,.
Then we have
\be\label{Lem412}
\lim_{\eps\downarrow 0}\frac{Z_d(\eps)}{Y_d(\eps)}=1\,.
\ee
\end{lem}

\noindent
There are some simple consequences of  Lemma \ref{combi} which are of interest for us.
Taking logarithms  in \eqref{Lem412} yields
$$
    \lim\limits_{\varepsilon \downarrow 0} \Big(\ln Z_d(\eps) - \ln Y_d(\eps)\Big) = 0.
$$
Since $\lim\limits_{\eps \downarrow 0} Y_d(\eps) = \infty$\,, we get
\be\label{f80}
  \lim\limits_{\eps\downarrow 0}\frac{\ln Z_d(\eps)}{\ln Y_d(\eps)} = \lim\limits_{\eps \downarrow 0}\frac{\ln Z_d(\eps)-\ln Y_d(\eps)}{\ln Y_d(\eps)}+1 = 1\,.
\ee
Clearly,
$$
\frac{\eps\cdot Z_d(\eps)}{(\ln Z_d(\eps))^{d-1}} =
\frac{\eps\cdot Y_d(\eps)}{(\ln Y_d(\eps))^{d-1}}\cdot
\left(\frac{\ln Y_d(\eps)}{\ln Z_d(\eps)}\right)^{d-1}\cdot
\frac{Z_d(\eps)}{Y_d(\eps)}\,,
$$
and together with

\[
  \lim_{\eps\downarrow 0}\frac{\eps\cdot Y_d(\eps)}{(\ln Y_d(\eps))^{d-1}} =  \frac{2^d}{(d-1)!} \,,
\]
see \cite[formula (4.23)]{KSU2},
\eqref{Lem412} and \eqref{f80} this implies
\be\label{ws-20_b}
  \lim_{\eps\downarrow 0}\frac{\eps\cdot Z_d(\eps)}{(\ln Z_d(\eps))^{d-1}}
= \lim_{\eps\downarrow 0}\frac{\eps\cdot Y_d(\eps)}{(\ln Y_d(\eps))^{d-1}}
=\frac{2^d}{(d-1)!}\,.
\ee
We will use these arguments with $z_\ell := (1+|\ell|^q)^{-1/q}$ (usual modification if $q=\infty$).
Rephrasing \eqref{ws-20_b} in the language of the numbers $C(r,\vec{1}, \vec{q})$
we conclude

\begin{equation}\label{ws-19}
\lim_{r \to \infty} \frac{C(r,\vec{1}, \vec{q})}{r(\ln C(r,\vec{1}, \vec{q}))^{d-1}} =
\frac{2^d}{(d-1)!} \, .
\end{equation}
Next we employ Lemma \ref{an1}. Hence, \eqref{ws-19} yields
\[
\lim_{m\to\infty}  \frac{n_m\,
a_{n_m} (I_d: \, H^{\vec{1},\vec{q}}_\mix (\T)\to L_2 (\T))}{(\ln n_m)^{(d-1)}} =
\lim_{m \to \infty} \frac{C(\vartheta_m,\vec{1}, \vec{q})}{\vartheta_m\, (\ln C(\vartheta_m,\vec{1}, \vec{q}))^{d-1}} =
\frac{2^d}{(d-1)!} \, .
\]
To deal with the general case we shall use  simple monotonicity properties.
The function $C(r,\vec{1}, \vec{q})$ is increasing in $r$ and tends to $\infty$ if $r$ tends to $\infty$.
Hence, for $r \ge r_0$
the function $C(r,\vec{1}, \vec{q})/(\ln C(r,\vec{1}, \vec{q}))^{d-1}$ is increasing as well, at least
for $r$ sufficiently large.
Let $n_m < n \le n_{m+1}$. We obtain
\[
 \frac{n\,
a_{n} (I_d: \, H^{\vec{1},\vec{q}}_\mix (\T)\to L_2 (\T))}{(\ln n)^{(d-1)}}
\le
\frac{C(\vartheta_{m+1},\vec{1}, \vec{q})}{\vartheta_{m+1} (\ln C(\vartheta_{m+1},\vec{1}, \vec{q}))^{(d-1)}}
\]
and
\[
\frac{C(\vartheta_{m},\vec{1}, \vec{q})}{\vartheta_{m+1} (\ln C(\vartheta_{m},\vec{1}, \vec{q}))^{(d-1)}}
\le \frac{n\,
a_{n} (I_d: \, H^{\vec{1},\vec{q}}_\mix (\T)\to L_2 (\T))}{(\ln n)^{(d-1)}} \, .
\]
Obviously
\[
 \lim_{m\to \infty} \frac{\vartheta_m}{\vartheta_{m+1}} = 1\, .
\]
As a consequence we find
\[
\lim_{n\to\infty}  \frac{n\,
a_{n} (I_d: \, H^{\vec{1},\vec{q}}_\mix (\T)\to L_2 (\T))}{(\ln n)^{(d-1)}} =
\frac{2^d}{(d-1)!} \, .
\]
{\em Step 2.}  Let $\vec{q}$ be a constant vector generated by some $q\in (0,\infty]$.
Let $\vec{s}:= (s, \ldots \, , s)$ for some $s>0$.
Then the claim follows from Step 1 and \eqref{eq073}.
\\
{\em Step 3.}
Let $\vec{q}= (q_1, \ldots \, , q_d)$, $q_j \in (0, \infty]$, $j = 1, \ldots \, , d$.
Let $\vec{s}:= (s, \ldots \, , s)$ for some  $s>0$.
We define
\[
 \min_{j=1, \ldots \, , d} q_j = \gamma_0 \qquad \mbox{and}\qquad \max_{j=1, \ldots \, , d} q_j = \gamma_1
 \, .
\]
By $\vec{\gamma_0}$ and $\vec{\gamma_1}$
we denote the constant vectors generated by $\gamma_0 $ and $\gamma_1$, respectively.
Next we shall use the chain of inequalities
\[
\|\, f \, |H^{\vec{s},\vec{\gamma_1}}_\mix (\T)\| \le \| \, f \, | H^{\vec{s},\vec{q}}_\mix (\T)\|
\le \| \, f \, | H^{\vec{s},\vec{\gamma_0}}_\mix (\T)\|\, .
\]
By Step 1 and Step 2 we know the asymptotic behavior of
$a_{n} (I_d: \, H^{\vec{s},\vec{\gamma_i}}_\mix (\T)\to L_2 (\T))$, $i=1,2$.
From the multiplicativity of the approximation numbers, see also Lemma \ref{back},
we finally conclude \eqref{nochwas}.
\end{proof}

Our goal is to extend the previous result to the situation of non-constant smoothness vectors. The core
observation is given in the following lemma.

\begin{lem}\label{key}
 Let $a= (a_n)_{n=1}^\infty, b= (b_n)_{n=1}^\infty \in c_0$ be two sequences of positive real numbers with limit zero,
 and let $(c_n)_{n=1}^\infty $ denote the
non-increasing rearrangement of the tensor product sequence
$$
a\otimes b =(a_j\,b_k)_{j,k\in\N}\,.
$$
Then
$$
\lim_{j\to\infty}\frac{j^\beta\,a_j}{(\log j)^\alpha}= \lambda
$$
for some constants $\beta, \lambda>0$ and $\alpha \ge 0$ implies
$$
\lim_{n\to\infty}\frac{n^\beta\,c_n}{(\log n)^\alpha} = \lambda\,\Big( \sum_{k=1}^\infty b_k^{1/\beta} \Big)^\beta\,.
$$
Note that the statement remains true even if $b \notin \ell_{1/\beta}$ since then 
$\lim_{n\to\infty}\frac{n^\beta\,c_n}{(\log n)^\alpha} = \infty$.
\end{lem}

\begin{proof}
{\em Step 1.} Reduction to the case $\beta = 1$.
It is enough to prove the claim for $\beta=1$. 
Indeed, if we define $\widetilde{a}_j := a_j^{1/\beta}$ and $\widetilde{b}_k := b_k^{1/\beta}$,
then the assumption on $(a_j)_j$ gives by taking power $1/ \beta$
\[
\lim_{j\to\infty}\frac{j\,\widetilde{a}_j}{(\log j)^{\alpha/\beta}}= \lambda^{1/\beta}\, .
\]
Now the case $\beta=1$ implies for the non-increasing rearrangement $(\widetilde{c}_n)_{n=1}^\infty$ of 
\[
 \widetilde{a} \otimes \widetilde{b} = (\widetilde{a}_j\, \widetilde{b}_k)_{j,k \in \N} = ({a}_j^{1/\beta}\, {b}_k^{1/\beta})_{j,k \in \N}
\]
that
\[
 \lim_{n\to\infty}\frac{n\,\widetilde{c}_n}{(\log n)^{\alpha/\beta}} = \lambda^{1/\beta}\,\sum_{k=1}^\infty \widetilde{b_k} \,.
\]
Since $\widetilde{c}_n = c_n^{1/\beta}$ and $\widetilde{b}_k = b_k^{1/\beta}$ this is equivalent (by taking power $\beta$) to 
\[
 \lim_{n\to\infty}\frac{n^\beta\,c_n}{(\log n)^\alpha} = \lambda\,\Big( \sum_{k=1}^\infty b_k^{1/\beta} \Big)^\beta\,.
\]
{\em Step 2.} The case $\beta =1$.
Since the set of products $a_j \, \cdot \, b_k$ does not depend on the ordering of the sequences
we may assume that both sequences are ordered, i.e.,
\[
a_1 \ge a_2 \ge \ldots \, \ge a_j \ge \ldots \qquad \mbox{and}\qquad
b_1 \ge b_2 \ge \ldots \, \ge b_k \ge \ldots \, .
\]
By homogeneity of the claimed  assertion we may also assume that $a_1=b_1=1$.
In what follows we use the notation $x_n \asymp y_n $ if $\lim_{n\to \infty} x_n/y_n = 1$.
We begin with a simple observation. Let
$$
N(\eps):=\#\{j\in\N: a_j\ge \eps\}\quad,\quad \eps >0\,.
$$
Then, for arbitrary $C>0$ and $\alpha \ge 0$, it is easy to verify that
$$
a_j\asymp\frac Cj (\log j)^\alpha \quad\text{as } j\to\infty\quad\Longleftrightarrow\quad
N(\eps)\asymp\frac C\eps(\log \frac 1\eps)^\alpha \quad\text{as } \eps\to 0\,.
$$
By our assumption and this observation we have
\[
N(\eps)\asymp \frac{\lambda}{\eps}\Big(\log \frac 1\eps\Big)^\alpha=:\varphi(\eps)\quad\text{as } \eps\to 0\,.
\]
We can apply this observation also to the quantities
\[
M(\eps):=\#\{n\in\N: c_n\ge \eps\}=\#\{(j,k)\in\N^2: a_j\,b_k\ge \eps\}\,,
\]
and therefore it is enough to show that
\[
\lim_{\eps\to 0}\frac{M(\eps)}{\varphi(\eps)}=\Vert b\Vert_1\,.
\]
\\
{\em Substep 2.1.} Lower estimate. Clearly we have
\be\label{Meps1}
M(\eps)=\sum_{k=1}^\infty \#\Big\{j\in\N: a_j\ge \frac{\eps}{b_k}\Big\}=\sum_{k\in\N:~ b_k\ge \eps}N(\eps/b_k)\,,
\ee
since $N(\eps/b_k)=0$ if $\eps/b_k>a_1=1$.
Fix now $m\in\N$. Then
$$
\liminf_{\eps\to 0}\frac{M(\eps)}{N(\eps)}\ge
\lim_{\eps\to 0}\sum_{k=1}^m\frac{N(\eps/b_k)}{N(\eps)}=\sum_{k=1}^m b_k\,,
$$
where we have used that
$$
\lim_{\eps\to 0}\frac{N(\lambda \eps)}{N(\eps)}=
\lim_{\eps\to 0}\frac{\varphi(\lambda \eps)}{\varphi(\eps)}=
\frac 1\lambda\quad\text{for every }\lambda >0\,.
$$
Letting $m\to\infty$ implies
$$
\liminf_{\eps\to 0}\frac{M(\eps)}{\varphi(\eps)}=\liminf_{\eps\to 0}\frac{M(\eps)}{N(\eps)}\ge\Vert b\Vert_1\,.
$$
\\
{\em Substep 2.2.} A preparation. We claim that the cardinalities
\[
 B(\varepsilon) := \# \{k \in \N:~ b_k \ge \varepsilon\}\, , \qquad 0 < \varepsilon \le 1\, , 
\]
satisfy
\[
 \lim_{\varepsilon \to 0}\, \varepsilon \, B(\varepsilon) = 0\, .
\]
For convenience of the reader we give a proof. 
Because of the monotonicity of the $b_k$ we have for  all $n \in \N$  
\[
 n\, b_{2n}\le \sum_{k=n+1}^{2n} b_k \le \sum_{k=n+1}^\infty  b_k \, \xrightarrow[n \to \infty]{}~ 0\, , 
\]
whence $\lim_{n \to \infty}\, n \, b_n = 0$, i.e., for all $\delta \ge 0$ $\exists n_\delta \in \N$ such that for all 
$n\ge n_\delta $ it holds $n\, b_n \le \delta$.
Let now $\delta>0$ and $0< \varepsilon \le 1$ be fixed  and $b_n \ge \varepsilon$.
Then either $n \le n_\delta $ of $n>n_\delta$ and $n \, \varepsilon \le n\, b_n \le \delta$.
This implies 
\[
 B(\varepsilon)\le \max\Big(n_\delta, \frac{\delta}{\varepsilon}\Big)
\]
and consequently 
\[
 \limsup_{\varepsilon \to 0}\, \varepsilon \, B(\varepsilon) \le \delta\, .
\]
Since this is true for every $\delta >0$, we obtain $\lim_{\varepsilon \to 0}\, \varepsilon \, B(\varepsilon) = 0$.
\\
{\em Substep 2.3} Upper estimate.
Fix any $\delta >0$ and select $\eps(\delta)>0$ such that
\be\label{Neps}
N(\eps)\le (1+\delta)\varphi(\eps)
\quad\text{for all}\quad 0<\eps\le\eps(\delta)\,.
\ee
We shall estimate $M(\eps)$ via formula (\ref{Meps1}). From  $b_k\ge \eps$ and
 $k\,b_k\le\Vert b\Vert_1$ we conclude that the number of summands in (\ref{Meps1}) is equal to  $B(\varepsilon)$.
If $\eps/b_k >\eps(\delta)$, we have
\[
N(\eps/b_k)\le N(\eps(\delta))\,,
\]
and if $\eps/b_k \le\eps(\delta)$, we estimate
\[
N(\eps/b_k)\le (1+\delta)\frac{\lambda\,  b_k}{\eps}\Big(\log \frac{b_k}{\eps}\Big)^\alpha
\le (1+\delta)\frac{\lambda\,  b_k}{\eps}\Big(\log \frac{1}{\eps}\Big)^\alpha \,,
\]
where we used the monotonicity of the $b_n$ and $b_1 = 1$.
Altogether this implies
\beqq
\frac{M(\eps)}{\varphi(\eps)} & \le &  \frac{B(\varepsilon)\cdot  N(\eps(\delta))}{\varphi(\eps)}+ (1+\delta)\, \Vert b\Vert_1\,
\le \varepsilon \, B(\varepsilon) \, \frac{N(\eps(\delta))}{\lambda\,(\log \frac 1\eps)^\alpha}+(1+\delta)\, \Vert b\Vert_1
\\
& \xrightarrow[\varepsilon \to 0]{}& (1+\delta)\Vert b\Vert_1\, , 
\eeqq
where we used Substep 2.2.
Hence
\[
\limsup_{\eps\to 0}\frac{M(\eps)}{\varphi(\eps)}\le (1+\delta)\Vert b\Vert_1\,.
\]
Since this is true for all $\delta >0$, the proof is finished.
\end{proof}

Combining Proposition \ref{lim2} and Lemma \ref{key} we arrive at the first main result.

\begin{satz}\label{lim1}
 Let $d\ge 2$ and $\vec{q}:= (q_1, \ldots \, , q_d)$, $q_j \in (0,\infty]$, $j=1, \ldots\, , d$.
Let $\vec{s}$ be given by
\[
0 < s_1 = s_2 = \ldots = s_\nu <  s_{\nu +1} \le  \ldots \le s_d< \infty 
\]
for some $\nu$, $1 \le  \nu < d$.
Then
\be\label{ws13}
 \lim_{n\to\infty}  \frac{n^{s_1}\,a_n (I_d: \, H^{\vec{s},\vec{q}}_\mix (\T)\to L_2 (\T))}{(\ln n)^{(\nu-1)s_1}}
= \Big[\frac{2^\nu}{(\nu - 1)!}\, \prod_{j= \nu+1}^d B_j\Big]^{s_1}\, ,
\ee
where
\[
B_j := 1+ 2 \sum_{m=1}^\infty (1+ m^{q_j})^{-\frac{s_j}{s_1\, q_j}}\, , \qquad j=\nu + 1 \, , \ldots \, d\, .
\]
\end{satz}

\begin{proof}
We shall proceed by induction. Therefore we fix $\nu \in \N$ and the induction runs with respect to $d$.
First, we investigate the case $d= \nu + 1$.
In this situation we choose
\[
a_j := a_j (I_d: ~ H^{(s_1, \ldots \, s_1),\vec{q}}_\mix (\tor^\nu) \to L_2 (\tor^\nu))\, , \qquad j \in \N\, .
\]
Proposition \ref{lim2} yields
\[
 \lim_{j \to \infty} \, \frac{j^{s_1}\,a_{j} (I_d: \, H^{(s_1, \ldots \, s_1),\vec{q}}_\mix (\tor^\nu) \to L_2 (\tor^\nu))}{(\ln j)^{(\nu-1)s_1} \, } =
 \Big[\frac{2^{\nu}}{(\nu-1)!}\Big]^{s_1} \, .
\]
Furthermore we know that
$(a_j)_j $ coincides with the non-increasing rearrangement of the sequence
\[
\prod_{\ell=1}^\nu (1+|k_\ell|^{q_\ell})^{-s_1/q_\ell}\, , \qquad k \in \zz^\nu \, .
\]
We choose
\[
b_k := (1+|k|^{q_{\nu +1}})^{-s_{\nu + 1}/q_{\nu +1}}\, , \qquad k\in \zz\, .
\]
The tensor product sequence  is given by
\[
 \prod_{\ell=1}^{\nu + 1} (1+|k_\ell|^{q_\ell})^{-s_\ell/q_\ell}\, , \qquad k \in \zz^{\nu+1} \, .
\]
The non-increasing rearrangement of this sequence, denoted by $c_n$, coincides with the approximation numbers of the identity operator with respect to
$(H^{\vec{s},\vec{q}}_\mix (\tor^{\nu+1}) \to L_2 (\tor^{\nu+1}))$. Hence,
Lemma \ref{key} with $\beta = s_1$, $\alpha =(\nu-1)s_1$ and $\lambda = \Big[\frac{2^{\nu}}{(\nu-1)!}\Big]^{s_1}$ implies
\be\label{wn-11}
 \lim_{n \to \infty} \frac{n^{s_1} \, \cdot \, c_n}{(\log n)^{(\nu-1)s_1}} = 
 \Big(1+ 2 \sum_{k=1}^\infty (1+|k|^{q_{\nu +1}})^{-\frac{s_{\nu + 1}}{s_1 q_{\nu +1}}} \Big)^{s_1} \, \Big[\frac{2^{\nu}}{(\nu-1)!}\Big]^{s_1} \, .
\ee
This coincides with \eqref{ws13} in our special situation.
The formula \eqref{wn-11} will serve as the initial step of our induction.
However, in a completely similar way we may prove the step $\nu + m \to \nu + m+1$ for arbitrary $m \in \N$.
The proof is finished.
\end{proof}

\begin{rem}
 \rm
(i) Clearly, all $B_j$ are finite.\\
(ii) The fine-index $\vec{q}$ and the ``jumps'' $s_j/s_\nu$, $j=\nu+1,...,d$, influence the asymptotic
behavior (in contrast to the classical mixed case with constant smoothness vector).
\\
(iii) If $s_{\nu +1} \downarrow s_\nu$ then $B_{\nu + 1}$ tends to infinity.
This reflects the following.
In case
\[
0 < s_1 = s_2 = \ldots = s_\nu =  s_{\nu +1} < s_{\nu+2} \le  \ldots \le s_d< \infty
\]
we know
\[
 \lim_{n\to\infty}  \frac{n^{s_1}\,a_n (I_d: \, H^{\vec{s},\vec{q}}_\mix (\T)\to L_2 (\T))}{(\ln n)^{\nu s_1}}
= \Big[\frac{2^{\nu+1}}{\nu!}\, \prod_{j= \nu+2}^d B_j\Big]^{s_1}\, ,
\]
see \eqref{ws13}. Compared to \eqref{ws13} the power of the logarithm has changed from $(\nu-1) s_1$
to $\nu s_1$. Hence, for $s_{\nu +1} \downarrow s_\nu$ the right-hand side in \eqref{ws13} must approach
infinity.
\\
(iv)
If $s_{\nu+1} \to \infty$, then $B_j \to 1$ for all $j\ge \nu + 1$ follows,
i.e., we are back in the $\nu$-dimensional case.
The approximation numbers $a_n$, in some sense, do not see the variables $x_{\nu + 1}, \ldots \, , x_d$,
if $n$ tends to infinity.
\\
(v) There is a general estimate from above for the asymptotic constants
$$
   \le  \Big[\frac{2^{\nu}}{(\nu-1)!}\Big]^{s_1}\,
  \Big[\prod_{j=\nu + 1}^d (1+ 2\,  \zeta ({s_j /s_1}))\Big]^{s_1}\,,
$$
where $\zeta$ denotes Riemann's zeta function, see \eqref{Riemann} below.
This implies that the constant $C_{\vec{s},\vec{q}}(d,\nu)$ in \eqref{ws-06}
decays super-exponentially in $\nu$ if $n$ is chosen sufficiently large.
This observation can be compared to similar results in  Bungartz, Griebel \cite{BG04},
Griebel \cite{Gr},  Schwab et al.\ \cite{SST08}, Dinh D{\~u}ng, Ullrich \cite{DiUl13}, Chernov, Dinh D{\~u}ng \cite{CD},
Krieg \cite{DK} and \cite{KSU2}, where all these references are dealing with the case $\nu=d$.
\end{rem}

In a similar way we can deal with the approximation  numbers of the anisotropic mixed Sobolev spaces
$a_n (I_d: \, H^{\vec{m}}_\mix (\T)\to L_2 (\T))$, $\vec{m}\in \N^d$, see \eqref{t-1}.

\begin{satz}\label{sobolev1}
 Let $d\ge 2$ and $1\le \nu < d$.
Let $\vec{m}\in \N^d$ be given by
\[
1 \le  m_1 = m_2 = \ldots = m_\nu <  m_{\nu +1} \le  \ldots \le m_d \,.
\]
Then
\beqq
 \lim_{n\to\infty}\, && \hspace{-0.5cm}  \frac{n^{m_1}\,a_n (I_d: \, H^{\vec{m}}_\mix (\T)\to L_2 (\T))}{(\ln n)^{(\nu-1)m_1}}
\\
& = & 
\bigg[\frac{2^\nu}{(\nu - 1)!}\, \prod_{j= \nu+1}^d \Big( 1 + 2\, \sum_{\ell=1}^\infty\, \frac{1}{(1 + \ell^2  + \ldots + \ell^{2m_j})^{1/(2m_1)}}
\Big)\bigg]^{m_1}\, .
\eeqq
\end{satz}

\begin{proof}
Our main tools are Lemma \ref{key} and Corollary 4.21 in \cite{KSU2}. 
The quoted corollary yields
 \be\label{neu1b}
 \lim_{n \to \infty} \, \frac{n^{m_1}\,a_{n} (I_\nu: \, H^{(m_1, \ldots \, ,m_1)}_\mix (\tor^\nu) \to L_2 (\tor^\nu))}{(\ln n)^{(\nu-1)m_1} \, } =
 \Big[\frac{2^{\nu}}{(\nu-1)!}\Big]^{m_1} \, .
 \ee
Temporarily we assume  $d = \nu + 1$. From \eqref{eq072}
we conclude that \\
$a_n (I_d: \, H^{\vec{m}}_\mix (\T)\to L_2(\T))$
equals the $n$-th element in the non-increasing rearrangement of the tensor product 
sequence
\[
 \Big(\prod_{j=1}^\nu \omega_{1}(k_j)\Big)^{-1}\, (\omega_{m_d}(k_d))^{-1} \, ,  \qquad k \in \zz^d\, , 
\]
see \eqref{neu2}. Let $(a_n)_{n=1}^\infty$ denote the
non-increasing rearrangement of the sequence
$\Big(\prod_{j=1}^\nu \omega_{1}(k_j)\Big)^{-1}$, $(k_1, \ldots \, , k_\nu) \in \zz^\nu$. In addition, let $(b_n)_{n=1}^\infty$ be the
non-increasing rearrangement of the sequence $(\omega_{m_d}(k))^{-1}$, $k \in \zz$. 
By $(c_n)_{n=1}^\infty$ we denote the non-increasing rearrangement of the tensor product sequence. 
Then \eqref{neu1b} yields
\[
 \lambda := \lim_{n \to \infty} \, \frac{n^{m_1}\,
 a_{n}}{(\ln n)^{(\nu-1)m_1}} =
 \Big[\frac{2^{\nu}}{(\nu-1)!}\Big]^{m_1} \, .
 \]
An application of Lemma \ref{key} with this positive number $\lambda$, $\beta:= m_1$ and $\alpha := (\nu-1)m_1$ leads to 
\beqq
 \lim_{n \to \infty} \, \frac{n^{m_1} \, c_n}{(\log n)^{(\nu-1)m_1}} & = & \lambda\, 
\Big(\sum_{n=1}^\infty b_n^{1/m_1}\Big)^{m_1} = \lambda \, \Big(\sum_{\ell\in \zz}\, \omega_{m_d}(\ell)^{-1/m_1}\Big)^{m_1}
\\
& = &  \lambda\, \Big( 1 + 2\, \sum_{\ell=1}^\infty\, \frac{1}{(1 + \ell^2  + \ldots + \ell^{2m_d})^{1/(2m_1)}}\Big)^{m_1}\, .
\eeqq
This proves the claim in our special case   $d = \nu +1$.
The induction step is using the same type of arguments, we will not repeat this.
\end{proof}

\begin{rem}\rm
We add a comment.
Observe that the sequence of approximation numbers 
$a_n (I_d: \, H^{\vec{m}}_\mix (\T)\to L_2 (\T))$
does not have the property 
\[
    a_n(I_d:H^{\lambda \vec{m}}_{\mix}(\T)\to L_2 (\T)) = a_n(I_d:H^{\vec{m}}_{\mix}(\T)\to L_2(\T))^\lambda \, , \qquad \lambda \in \N\, , 
 \]
compare with \eqref{eq073}. 
This follows from 
\[
 \omega_{m_1}(\ell)) \, \cdot \, \omega_{m_2}(\ell)) \neq \omega_{m_1+m_2}(\ell)\, , \qquad \ell \neq 0\, .
\]
\end{rem}


\section{Preasymptotics}

\label{numbersmall}


Meanwhile it is well-known that in case of Sobolev embeddings
the qualitative structure of the bounds for the $a_n$ in case of large $n$ (the so-called asymptotic case) significantly differs
from those for small $n$ (the so-called preasymptotic case).
In this section we deal with the behavior of the approximation numbers for small $n$, i.e., $n\le c\, 2^d$,
where $c$ will be specified later.

After some preliminaries we will turn to a detailed discussion of the case of a constant vector $\vec{s}$.
This simplified situation will be used to discuss our  method.
Compared to our earlier investigations of this case, see \cite{KSU2},
this represents a certain progress in quality of the results as well as in the complexity of the used method (much more simple).
Next we will discuss the case of at least one  jump in the sequence $\vec{s}$. Finally,
we will deal with the case of logarithmically increasing smoothness.

What concerns the dependence on $\vec{q}$ we will proceed as follows.
Always it is quite easy to understand the extremal situation $\vec{q}= (\infty, \ldots \, , \infty)$.
Next we  investigate
the behavior of $a_n (I_d:\, H^{\vec{s},\vec{1}}_{\mix}(\T)\to L_2 (\T))$
(which is of basic importance for the general situation).
Afterwards we deal with
the behavior of $a_n (I_d:\, H^{\vec{s},\vec{q}}_{\mix}(\T)\to L_2 (\T))$ for general $\vec{q}$, which becomes a simple conclusion of 
the case $\vec{q}= (1, \ldots \, , 1)$.

Finally, we would like to direct the attention of the reader to the following.
In the preasymptotic range it does not make sense to speak about an optimal approximation rate.
Usually we will prove an estimate in the form
\[
 a_n (I_d:\, H^{\vec{s},\vec{q}}_{\mix}(\T)\to L_2 (\T)) \le C(d,s)\, n^{-\gamma (n,d)s}\, , \qquad 2\le n\le c\, 2^d\, .
\]
Any change of the rate $\gamma (n,d)s$ can be compensated by a change of the constant $C(d,s)$.
This means, when comparing results one has to discuss $C(d,s)\, n^{-\gamma (n,d)s}$ together.


\subsection{The crucial lemma}


Of basic importance for all what follows are  estimates of
$C(r,\vec{s},\vec{q})$.
Therefore we shall need Riemann's $\zeta$-function, i.e.,
\be\label{Riemann}
 \zeta (t):= \sum_{j=1}^\infty \frac{1}{j^t}\, , \qquad t>1\, .
\ee
It will be enough to consider the case $\vec{q}=\vec{1}$.
We put
\[
 c(r,d):= C(r, (s_1, \ldots \, s_d), (1, \ldots \, , 1))\, ,
\]
see \eqref{cardinal}.

\begin{lem}
 \label{clever}
Let $1= s_1\le s_2 \le \ldots \, \le s_\ell \le \ldots \, $.
Then, for any $\alpha > 1$
\be\label{wn-1}
c(r,\ell)\le A_\alpha \, \Big(\prod_{j=2}^\ell (2 \zeta (\alpha s_j)-1) \Big)\, r^\alpha
\ee
holds for all $r\ge 1$ and  all $\ell \ge 2$,
where
\be\label{Konstante}
A_\alpha := \sup_{r\ge 1} \, \frac{2 \lfloor r \rfloor-1}{r^\alpha}\, .
\ee
\end{lem}

\begin{proof}
 We proceed by induction on $\ell$.
 Let $\ell = 1$.  Then we have
 \[
  c(r,1) = 2 \lfloor r \rfloor-1 \le A_\alpha \, r^\alpha
 \]
We put $t:= s_{\ell + 1}$.
For the next step we shall use the representation formula
\[
 c(r, \ell + 1) = c(r,\ell) + 2\, \sum_{m=2}^{\lfloor r^{1/t} \rfloor} c(r/m^t, \ell)\ .
\]
To see this,  observe that
\beqq
\# \bigg\{k \in \zz^{\ell + 1}: && \hspace{-0.6cm}  \prod_{j=1}^{\ell +1} (1+|k_j|)^{s_j}\le r\bigg\} =
\# \bigg\{k \in \zz^{\ell + 1}:~ \prod_{j=1}^{\ell } (1+|k_j|)^{s_j}\le \frac{r}{(1+|k_{\ell + 1}|)^{t}}\bigg\}
\\
&=&  \# \bigg\{k \in \zz^{\ell}:~ \prod_{j=1}^{\ell} (1+|k_j|)^{s_j}\le r\bigg\}
\\
&  & \quad + \, 2\,  \sum_{m=2}^{\lfloor r^{1/t} \rfloor} \# \bigg\{k \in \zz^{\ell}:~ \prod_{j=1}^{\ell} (1+|k_j|)^{s_j}\le \frac{r}{m^{t}}\bigg\}\, .
\eeqq
By means of the representation formula it is now easy to prove \eqref{wn-1} in case $\ell =2$
which is our point of departure for the induction.
Formula \eqref{wn-1} represents our induction hypothesis.
Hence, if $\ell \ge 2$
\beqq
c(r, \ell + 1) & \le &  A_\alpha \, \Big(\prod_{j=2}^\ell (2 \zeta (\alpha s_j)-1) \Big) 
\, r^\alpha + 2 A_\alpha \, \Big(\prod_{j=2}^\ell (2 \zeta (\alpha s_j)-1) \Big)\,
\sum_{m=2}^{\lfloor r^{1/t} \rfloor} (r/m^t)^\alpha
\\
&\le &
A_\alpha \, \Big(\prod_{j=2}^\ell (2 \zeta (\alpha s_j)-1) \Big) \, r^\alpha \Big( 1 + 2 \sum_{m=2}^{\infty} \frac{1}{m^{t\alpha}}\Big)
\\
&=& A_\alpha \, \Big(\prod_{j=2}^\ell (2 \zeta (\alpha s_j)-1) \Big) \, r^\alpha \Big(2 \zeta (t\alpha) -1\Big)
\\
&=& A_\alpha \, \Big(\prod_{j=2}^{\ell+1} (2 \zeta (\alpha s_j)-1) \Big) \, r^\alpha
\eeqq
as claimed.
\end{proof}

\begin{rem}\label{berechnung}
 \rm
It is easy to check that $A_{\alpha} \leq 2$ for $\alpha >1$ and $A_{\alpha} = 1$ for $\alpha\geq \ln(3)/\ln(2) \approx 1.58496$.

%
\end{rem}

There exist many contributions in the literature dedicated to this problem
(estimates of $C(r,\vec{s},\vec{q})$).
Let us mention here at least \cite{KSW},  \cite{CKN} and \cite{DK}.
In \cite{KSW},  \cite{CKN}  the authors are dealing with the cardinality of weighted
Zaremba crosses, see \cite{CKN}, formula (6) on page 69.
Such a weighted Zaremba cross coincides with the set $C(r,\vec{1},\vec{\infty})$.
 In  \cite{KSW} the authors proved for arbitrary $\alpha > 1/s$ and all $r>0$ the inequality
  \beqq
&&C(r, \underbrace{(s, \ldots \, ,s)}, \underbrace{(\infty, \ldots \, , \infty)}) \le  \,
(2 \zeta (\alpha s)+1)^d \, r^\alpha \, .
\\
&& \qquad d \mbox{~~times}\qquad \quad d \, \mbox{times}
\eeqq
In \cite{CKN}  the authors obtained the following estimate in dimension
$d$ of the form
 \beqq
&&C(r, \underbrace{(1, \ldots \, 1)}, \underbrace{(\infty, \ldots \, , \infty)}) \le  \, (2 \zeta (\alpha)+1)^d \, r^\alpha
\\
&& \qquad d \mbox{~~times}\qquad \quad d \, \mbox{times}
\eeqq
for all $r \in \N$ and $\alpha > 1$. 
Observe, that in both estimates there is the number  $+1$ instead of the number 
$-1$ as in our case.
The change from $d$ to $d-1$ seems to be less important.


\subsection{A negative result for $a_n (I_d:\, H^{\vec{s},\vec{\infty}}_{\mix}(\T)\to L_2 (\T))$}\label{null}


This is the most simple case. As usual we assume $\min_{j=1, \ldots \, , d}\,  s_j >0$. Because of
\[
 C(1, {(s_1, \ldots \, , s_d)}, {(\infty, \ldots \, , \infty)})
 = \# \Big\{k \in \Z: \,~ \prod_{j=1}^d \max (1, |k_j|)^{s_j} \le 1 \Big\} = 3^d\, ,
\]
by applying \eqref{eq072},
we get the following conclusion.

\begin{satz}\label{unendl}
Let $d\in \N$. Let $\vec{s}=(s_1, \ldots \, , s_d)$ be a vector with positive components.
Then
\[
 a_n (I_d:\, H^{\vec{s},\vec{\infty}}_{\mix}(\T)\to L_2 (\T)) = 1\, , \qquad n=1, 2, \ldots \, , 3^d\, .
\]
\end{satz}

\begin{rem}
 \rm
 (i) This is, in some sense, a worst case. There is no approximation at all for $n\le 3^d$ (independent on $\vec{s}$).
 Only if we are able to spend more than $3^d$ pieces of information on the function $f$ we probably get an  approximation with an error $<1$.
 \\
 (ii)
  For isotropic Sobolev spaces a similar result has been obtained in  \cite{KMU}.
  \\
 (iii) In a nonperiodic context Novak and Wo{\'z}niakowski \cite{NoWo09}
 have proved a result in this spirit for approximation of smooth functions in $L_\infty$.
\end{rem}


\subsection{The preasymptotic decay in case of constant $\vec{s}$ revisited}\label{smalln1}


In this subsection we consider $\vec{s}= (s, \ldots \, s)$ (with a slight abuse of notation)
for some $s>0$. First we deal with the  case
$\vec{q}= (1, \ldots , 1) = \vec{1}$.


\paragraph{The behavior of the
$a_n$ for $\vec{q}= \vec{1}$.}


In \cite{KSU2} we already studied the preasymptotics in case $s_1 = \ldots = s_d>0$, see also
 Dinh D{\~u}ng, Ullrich \cite{DiUl13}, Chernov, Dinh {D}\~ung \cite{CD} and Krieg \cite{DK}.
Quite recently, in \cite{Ku}, one of the authors improved our result from \cite{KSU2} and that one of 
\cite{DK} and obtained the following.

\begin{prop}\label{small}
Let $\vec{s} = (s, \ldots \, s)$ for some $s>0$ and  $d\in \N$, $d\ge 2$.
For all $n \ge 6$ it holds
\be\label{ws-46}
a_n(I_d:H^{\vec{s}, \vec{1}}_{\mix}(\T) \to L_2(\T)) \leq
\Big(\frac{16}{3n}\Big)^{\frac{s}{1+\log_2 d}}\,.
\ee
\end{prop}

Now we shall present a simple method based on Lemma \ref{clever} which will allow us to improve
\eqref{ws-46} for $d$ large enough.
We employ the notation used in  Lemma \ref{an1} and  assume that
\be\label{neu1}
c(\vartheta_m,d) < n \le c(\vartheta_{m+1},d) = n_{m+1}\, .
\ee
Then $ a_n = 1/\vartheta_{m+1}$, see Lemma \ref{an1}. Next we use Lemma \ref{clever} with $r= \vartheta_{m+1}$.
For any $\alpha >1$ this yields
\beqq
 n & \le &  c(\vartheta_{m+1},d) \le A_\alpha \, \, \Big(\prod_{j=2}^\ell (2 \zeta (\alpha s_j)-1) \Big)\, \vartheta_m^\alpha
 =   A_\alpha \, \, \Big(\prod_{j=2}^d (2 \zeta (\alpha s_j)-1) \Big)\, a_n^{-\alpha}\, .
\eeqq
Hence 
\be
\label{neu1ab}
a_n \le A^{1/\alpha}_\alpha \, \, \Big(\prod_{j=2}^d (2 \zeta (\alpha s_j)-1) \Big)^{1/\alpha}\, n^{-1/\alpha}\, .
\ee
This estimate will be used below in various situations, not only in the following one. 
As a further preparation  we derive a simple estimate of $2 \zeta (t)-1 $.
Let $t >1$. Obviously it holds
\beq\label{wn-2}
2 \zeta (t)-1 & = & 1+ 2 \, \sum_{j=2}^\infty \frac{1}{j^t} \le 1 + 2^{-t+1} + 2\, \int_{2}^\infty x^{-t}\, dx
\nonumber
\\
& = &  1 + 2^{-t+1} + \frac{2^{2-t}}{t-1}=  1 + 2^{-t}\, \Big( 2+ \frac{4}{t-1}\Big)\, .
\eeq
Now we turn back to our problem. Let $d\ge 3$ and temporarily we assume $\vec{s}= \vec{1}$.
We choose $\alpha:= 1+ \log_2 (d-1)$. Lemma \ref{clever} and \eqref{wn-2} imply
\beqq
\frac{c(r,d)}{A_\alpha} & \le & \Big(2 \zeta (\alpha)-1 \Big)^{d-1}\, r^\alpha \\
&\le & \Big[ 1+ 2^{-(1+\log_2 (d-1))} \Big(2 + \frac{4}{\log_2 (d-1)}\Big)\Big]^{d-1}\, r^\alpha
\\
& = & \Big[ 1+ \frac{1}{d-1} \Big(1 + \frac{2}{\log_2 (d-1)}\Big)\Big]^{d-1}\, r^\alpha
\eeqq
We put
\be\label{konstante}
 C(d):= \Big[ 1+ \frac{1}{d-1} \Big(1 + \frac{2}{\log_2 (d-1)}\Big)\Big]^{d-1}\, , \qquad d\ge 3\, .
\ee
Taking into account $A_\alpha =1$ for $d\ge 3$ our previous estimate \eqref{neu1ab}
leads to 
\[
a_n \le \Big(\frac{C(d)}{n}\Big)^{1/\alpha}
\]
for all $n$ as in \eqref{neu1}.
 
\begin{prop}\label{smallbbb}
Let $\vec{s} = (s, \ldots \, s)$ for some $s>0$ and  $d\in \N$, $d\ge 3$.
For all $n \ge 2$ it holds
\be\label{ws-46bbb}
a_n(I_d:H^{\vec{s}, \vec{1}}_{\mix}(\T) \to L_2(\T)) \leq
\Big(\frac{C(d)}{n}\Big)^{\frac{s}{1+\log_2 (d-1)}}\,,
\ee
where $C(d)$ is defined in \eqref{konstante}.
\end{prop}

\begin{proof}
The case $\vec{s}=\vec{1}$ has been proved above.
The general case follows by an application of \eqref{eq073}.
\end{proof}

\begin{rem}
 \rm
Observe, that $C(d)$ is strongly decreasing and $\lim_{d\to \infty} \, C(d) = e$.
A rough estimate of $C(d)$ is given by
\[
 C(d)\le e^{1+ \frac{2}{\log_2 (d-1)}}\, .
\]
Here is a list of the first few values of $C(d)$
(to show the improvement in \eqref{ws-46bbb} compared to  \eqref{ws-46}):
\begin{center}
\begin{tabular}{|c|c|c|c|c|c|c|c|}\hline \hline
&&&&&&&\\[-3mm]
$d$ &  $C(d)$   & $d $&  $C(d)$ & $d$ &  $C(d)$ & $d$ & $C(d)$  \\[1mm]
\hline \hline
3 &  6.250  & 9 &  4.545   & 15 & 4.254  & 21&  4.103
\\
4 &  5.396 & 10 &  4.476  & 16 & 4.222 & 22 & 4.084
\\
5 & 5.063 &  11 &  4.419  & 17 & 4.195 & 23& 4.067
\\
6 & 4.866   &   12  &  4.370  & 18 & 4.169  & 24& 4.050
\\
7 & 4.730 & 13 &  4.326 & 19 & 4.145   & 25 & 4.034
\\
8 &  4.627 & 14 &  4.288 & 20 &  4.123 & 26 & 4.020
\\
\hline\hline
\end{tabular}
\end{center}
Hence, beginning with $d=5$ \eqref{ws-46bbb} is better than  \eqref{ws-46}.
\end{rem}

In \cite{DK} and \cite{KSU2} one can find also lower bounds.  We are quoting here the result from
Krieg \cite{DK}.

\begin{prop}\label{lowersmall}
Let $\vec{s} = (s, \ldots \, s)$ for some $s>0$ and  $d\in \N$, $d\ge 2$.
For $n \ge 3$ we define
\[
\gamma (n,d) :=  \log_2\Big( 1+ \frac{2d}{\log_3 n} \Big) \, .
\]
For all $3 \le  n \le 3^d$ it holds
\[
a_n(I_d:H^{\vec{s},\vec{1}}_{\mix}(\T) \to L_2(\T)) \geq
    2^{-s}\,  n^{-\frac{s}{\gamma (n,d)}}\,.
\]
\end{prop}

In the meanwhile we know that the structure of the estimate from below is closer to the optimal one
than the other.
The correct behavior is not reflected by a simple power in $n$, one needs more complicated
functions (those as in Proposition \ref{lowersmall}).
We proceed by the same method as in the proof of Proposition \ref{smallbbb} but changing $\alpha$ now.
Let $d,n\ge 2$ and
\[
\alpha = \alpha (n,\beta,d) := 1+ \log_2 \Big(2+ \frac{\beta (d-1)}{\ln n} \Big)\, ,
\]
where $\beta > 2 $ will be chosen later on.
This choice of $\alpha$ is based on the observation that our method will produce a
convergence rate close to
$1/\alpha$ for the case $1= s_1 =  s_2 = \ldots = s_d$.
We do not claim that this is the best choice.

As above,  we first deal with the simplified case  $1= s_1 =  s_2 = \ldots = s_d$.
Lemma \ref{clever} and \eqref{wn-2} imply
\beqq
\frac{c(r,d)}{A_\alpha} & \le & \Big(2 \zeta (\alpha)-1 \Big)^{d-1}\, r^\alpha \\
&\le & \Big[ 1+ 2^{-(1+\log_2 (2 + \frac{\beta (d-1)}{\ln n}))} \Big(2 + \frac{4}{\log_2 (2 + \frac{\beta (d-1)}{\ln n})}\Big)\Big]^{d-1}\, r^\alpha
\\
& = & \Big[ 1+ \frac{1}{2+ \frac{\beta (d-1)}{\ln n}} \Big(1 + \frac{2}{\log_2 (2 + \frac{\beta (d-1)}{\ln n})}\Big)\Big]^{d-1}\, r^\alpha
\eeqq
Since our aim consists in an investigation of the preasymptotic case  we may suppose
\be\label{klar}
2\le n \le e^{\beta (d-1)/2}\, .
\ee
Observe that
\[
 \log_2 \Big(2+ \frac{\beta (d-1)}{\ln n} \Big) \ge 2
 \qquad \Longleftrightarrow \qquad \frac{\beta (d-1)}{2} \ge \ln n\, .
\]
Hence
\beqq
\frac{c(r,d)}{A_\alpha} & \le &
\Big[ 1+ \frac{2 \ln n}{2\ln n + \beta (d-1)} \Big]^{d-1}\, r^\alpha
\\
& \le &
\Big[ 1+ \frac{2 \ln n}{\beta (d-1)} \Big]^{d-1}\, r^\alpha
\\
& \le &  e^{(2 \, \ln n)/\beta}\, r^\alpha = n^{2/\beta} \, r^\alpha \, .
\eeqq
Now we proceed exactly as above. In case $c(\vartheta_m,d) < n \le c(\vartheta_{m+1},d)$ and with $r= \vartheta_{m+1}$
we  obtain

\[
a_n \le \Big(\frac{A_\alpha \, n^{2/\beta}}{n}\Big)^{1/\alpha}
\, ,
\]
see \eqref{neu1ab}.
Our assumption \eqref{klar} yields  $A_\alpha =1$, see Remark \ref{berechnung}.
In the following Proposition we summarize our observations.

\begin{prop}\label{smallb}
Let  $d\ge 2$ and $\beta >2$. Let $\vec{s}= (s, \, \ldots\, , s)$ for some $s>0$.
For all $ n$, $2 \le n \le e^{d-1}$ it holds
\be\label{ws-46b}
a_n(I_d:H^{\vec{s},\vec{1}}_{\mix}(\T) \to L_2(\T)) \leq  \, n^{-\gamma s}\, ,
\ee
where
\be\label{gam}
\gamma = \gamma (n,\beta,d) := \frac{(1-2/\beta)}{1+ \log_2\Big(2+\frac{\beta (d-1)}{\ln n}\Big)}\, .
\ee
\end{prop}

\paragraph{Discussion of Proposition \ref{smallb}.}
 \rm
 Define $\delta:= 2/\beta$. Then
\be\label{besser}
 \gamma (n,\beta,d) > \frac{1}{1+ \log_2 (d-1)}
\ee
is equivalent to
\be\label{hin}
\delta \, \Big[\frac{(d-1)^{1-\delta}}{2^{1+\delta}} - 1\Big] > \frac{d-1}{\ln n} \, .
\ee
Hence, for any fixed $\delta \in (0,1)$, for sufficiently large $d$,  there is always a nonempty interval for $n$
(not too far away from $e^{d-1}$) such that \eqref{besser} is true. This general observation can be made more precise.
For given $\delta$ and $d$ sufficiently large we fix $\varepsilon \in (0, 1-\delta)$ such that
\be\label{erg1}
8 < \min\Big(\delta \, (d-1)^\varepsilon,  (d-1)^{1-\delta} \Big)\, .
\ee
Now we choose $\sigma:= \delta + \varepsilon$ and $n \ge 3$ such that
\[
 (d-1)^\sigma \le \ln n \le \frac{1}\delta \, (d-1)\, .
\]
It follows
\[
 \frac{d-1}{\ln n} \le (d-1)^{1-\sigma}= \frac{(d-1)^{1-\delta}}{(d-1)^{\varepsilon}}< (d-1)^{1-\delta} \, \frac \delta 8
\le \delta \, \Big[\frac{(d-1)^{1-\delta}}{2^{1+\delta}} - 1\Big]\, ,
\]
where we used \eqref{erg1} in the last step.
Hence, Proposition \ref{smallb} improves Proposition \ref{smallbbb} if
\[
 e^{(d-1)^{(2/\beta) + \varepsilon}} = e^{(d-1)^\sigma}  \le n \le e^{\frac{1}\delta \, \frac{d-1}{2}} = e^{\frac{\beta}2 \, (d-1)}\, ,
\]
and \eqref{erg1} is satisfied.

Of course, we  are interested in the maximum of $\gamma (n,\beta,d)$ for fixed $n$ and fixed $d$.
Therefore we will have a look at the derivative of $\gamma$. Then it follows
\[
 \frac{\partial \gamma}{\partial \beta} (n,\beta,d) =
\frac{2\beta ^{-2} \Big[1+ \log_2\Big(2+\frac{\beta (d-1)}{\ln n}\Big)\Big] -\Big(1-\frac{2}{\beta}\Big)
\frac{\frac{d-1}{\ln n}}{\ln 2\, \Big[2+ \frac{\beta (d-1)}{\ln n}\Big]}}{\Big[1+ \log_2\Big(2+\frac{\beta (d-1)}{\ln n}\Big)\Big]^2}\, .
\]
It will be convenient to write $n$ in the form
$n= e^{(d-1)/\kappa}$, $\kappa \ge 1$.
We define
\[
 F(\kappa,\beta):= 2 \Big[1+ \log_2\Big(2+ \beta \kappa\Big)\Big] -\Big(\beta^2- 2\beta\Big)
\frac{\kappa}{\ln 2\, \Big[2+ \beta \kappa\Big]}\, , \qquad (\kappa, \beta) \in [1,\infty)\times [2,\infty)\, .
\]
For fixed $\kappa$ we need to find $\beta$ such that  $ F(\kappa,\beta)=0$ since $ F(\kappa,\beta)=0$ implies \\
$\frac{\partial \gamma}{\partial \beta} (n,\beta,d)=0$.
Because of $F(\kappa,2)>0$ and $F(\kappa,\beta)< 0 $ if $\beta$ is large enough there exists always at least one such $\beta$.
However, there is no explicit formula for $\beta$.
By numerical calculations we obtain in case $\kappa =1$ the value
\[
 \beta \sim 9.59824 \, .
\]
This gives
\[
\gamma (e^{d-1},9.59824,d) = \frac{(1-2/9.59824 )}{1+ \log_2\Big(2 + 9.59824 \Big)} \sim 0.174528\, .
\]
We need to compare this result with the estimates obtained in Proposition \ref{smallbbb}.
It turns out that the estimate in Proposition \ref{smallbbb} can be written as
\[
a_{e^{d-1}} (I_d:H^{\vec{s},\vec{1}}_{\mix}(\T) \to L_2(\T)) \le e^{-\delta (d)\, s \, (d-1)}
\]
where

\begin{center}
\begin{tabular}{|c|c|c|c|c|c|c|c|}\hline \hline
&&&&&&&\\[-3mm]
$d$ &  $\delta (d)$   & $d$ &  $ \delta(d)$ & $d$ &  $\delta(d)$ & $d$ & $\delta(d)$  \\[0mm]
\hline \hline
3 &  0.042  & 18 &  0.180   & 21 & 0.175  & 24 &  0.170
\\
9 &  0.203 & 19 &  0.178  & 22 & 0.173 & 25 & 0.169
\\
17 & 0.182 &  20 &  0.176  & 23 & 0.171 & 26 & 0.167
\\
\hline\hline
\end{tabular}
\end{center}

\noindent
Hence, our optimization of $\beta$, which is expressed by the estimate \eqref{ws-46b} (see also \eqref{gam}) 
\[
a_{e^{d-1}}(I_d:H^{\vec{s},\vec{1}}_{\mix}(\T) \to L_2(\T)) \leq  \, e^{-0.174528 s (d-1)}
\]
leads to an improvement if $d\ge 21$.
We repeated this procedure for several values of $\kappa$, determining a sequence $\beta (\kappa)$ on this way described in the following table.
\begin{center}
\begin{tabular}{|c|c|c|c|c|c|c|c|}\hline \hline
&&&&&&&\\[-3mm]
$\kappa$ &  $\beta (\kappa)$   & $\kappa$  & $\beta (\kappa)$ & $\kappa$ & $\beta (\kappa)$  & $\kappa$ & $\beta (\kappa)$   \\[0mm]
\hline \hline
1 &  9.60  & 5 &  67.60   & 9 & 143.69  & 50 &  1168.94
\\
2 &  20.72 & 6 &  85.58  & 10 & 164.15 & 70 & 1738.35
\\
3 & 34.77 &  7 &  104.33  & 20 & 388.12 & 100 & 2637.18
\\
4 & 50.58   & 8  &  123.73  & 30 & 634.94  & 500 & 16 628.70
\\
\hline\hline
\end{tabular}
\end{center}
{~}\\
For $\kappa \in \{1,2,\ldots \, , 10\}$
we may fit $\beta(k)$ (empirically) via
\[
\beta (\kappa) \sim (4 \kappa + 1)^{11/8}\,.
\]
Now we insert this formula into  \eqref{gam} and obtain
\be\label{gammaa}
\gamma^*=
\gamma^* \Big(e^{(d-1)/\kappa}, (4 \kappa + 1)^{11/8}\Big):= \frac{1- \frac{2}{(4 \kappa + 1)^{11/8}}}{1+
\log_2\Big(2+ (4 \kappa + 1)^{11/8} \kappa\Big)}\, .
\ee
In case $\kappa =1$ this yields
\[
 \gamma^* \Big(e^{d-1}, 5^{11/8}\Big) \sim  0.174462
\]
which is only slightly worse than the rate $\gamma (e^{d-1},9.59824,d)  \sim 0.174528$.

\begin{prop}\label{smallbcb}
Let $d\ge 7$ and  $\vec{s}= (s, \, \ldots\, , s)$ for some $s>0$. Let $\kappa := (d-1)/\ln n$, $n \ge 2$.
For all $ n$, $2 \le n \le e^{d-1}$ it holds
\be\label{ws-46bcb}
a_n(I_d:H^{\vec{s},\vec{1}}_{\mix}(\T) \to L_2(\T)) \leq  \, n^{-\gamma^* s}\, ,
\ee
where $\gamma^*$ is defined in \eqref{gammaa}.
At least, if
\be\label{gammaaa}
4\, (d -1)^{3/5} \le \ln n \le \frac{2\, (d-1)}{7}  \, ,
\ee
we have $\gamma^* > \frac{1}{1+ \log_2 (d-1)}$, i.e., \eqref{ws-46bcb} improves \eqref{ws-46bbb}.
\end{prop}

\begin{proof}
 Inequality \eqref{ws-46bcb} is a direct consequence of \eqref{ws-46b} with
$\beta$ chosen as $(4 \kappa + 1)^{11/8}$.
It will be enough to prove sufficiency of \eqref{gammaaa}.
The inequality $\gamma^* > \frac{1}{1+ \log_2 (d-1)}$ is equivalent to
\be\label{immer}
 \log_2 (d-1) >  \frac{2 (1+ \log_2 (d-1))}{(4 \kappa + 1)^{11/8}} +  \log_2\Big(2+ (4 \kappa + 1)^{11/8} \kappa\Big)\, .
\ee
Looking at this inequality then it becomes clear we need a lower bound for $\kappa$ (estimate of the first summand on the right-hand side)
and an upper bound for $\kappa$ (estimate of the second summand on the right-hand side).
Our assumptions in \eqref{gammaaa} can be rewritten as
\[
 \frac{7}{2} \le \kappa \le \frac{(d-1)^{2/5}}{4}\, .
\]
Because of  $d\ge 7$ we conclude $(d-1)^{2/5} \ge 2$. Since $\kappa > 3$ we find
\beqq
 2+ (4 \kappa + 1)^{11/8} \kappa & \le &  2 \, \kappa\,  (4 \kappa + 1)^{11/8}
 \le \frac{(d-1)^{2/5}}{2}\, \Big( (d-1)^{2/5} + 1 \Big)^{11/8}
 \\
 &\le & \frac{(d-1)^{2/5}}{2}\, \Big(\frac 32 \,  (d-1)^{2/5}\Big)^{11/8}
 \\
 & = & \frac 12 \Big(\frac 32\Big)^{11/8} \, (d-1)^{19/20}\, .
\eeqq
Hence, we get
\[
 \log_2\Big(2+ (4 \kappa + 1)^{11/8} \kappa\Big) \le \frac{19}{20}\,  \log_2 (d-1) + \log_2 \frac 12 \Big(\frac 32\Big)^{11/8}  \, .
\]
Now we turn to the other summand in \eqref{immer}. The inequality
\[
 \kappa \ge \frac{40^{8/11}-1}{4} \sim 3.4066
\]
is equivalent to
\[
 \frac{2}{(4 \kappa + 1)^{11/8}} \le \frac{1}{20} \, .
\]
Therefore we obtain
\beqq
 \frac{2 (1+ \log_2 (d-1))}{(4 \kappa + 1)^{11/8}} & + &  \log_2\Big(2+ (4 \kappa + 1)^{11/8} \kappa\Big)
\\
&<& \frac{1}{20} + \frac{1}{20}\,  \log_2 (d-1) +
\frac{19}{20}\,  \log_2 (d-1) + \log_2 \frac 12 \Big(\frac 32\Big)^{11/8}  \, .
\eeqq
Since $\frac{1}{20} + \log_2 \frac 12 \Big(\frac 32\Big)^{11/8} < 0$ our claim in \eqref{immer} follows.
\end{proof}


\paragraph{The behavior of the $a_n$ with $\vec{q} \neq \vec{1}$.}
\label{einsb}


Here we suppose $\vec{q} \neq \vec{1}$, more exactly, we shall assume that $\vec{q}$ is a constant vector
generated by some $q>1$.  Clearly, by a generalization of \eqref{ws-unendlich} and Lemma \ref{back}
we have some monotonicity of the approximation numbers, i.e.,
\[
 a_n(I_d:H^{\vec{s}, \vec{1}}_{\mix}(\T) \to L_2(\T))\le a_n(I_d:H^{\vec{s}, \vec{q}}_{\mix}(\T) \to L_2(\T)) \le 
a_n(I_d:H^{\vec{s}, \vec{\infty}}_{\mix}(\T) \to L_2(\T))
\]
holds for all $n$.

Now we extend the results of Propositions \ref{smallbbb} and \eqref{ws-46b} to constant vectors $\vec{q}$ generated by any $q \ge 1$.

\begin{satz}\label{smalldd}
Let $\vec{s} = (s, \ldots \, s)$ for some $s>0$.
 Let  $\vec{q}$ be the constant vector generated by some finite $q \geq {1}$.
\\
{\rm (i)} Let  $d\ge 3$.
For all $n \ge 2$ it holds
\be\label{ws-46db}
a_n(I_d:H^{\vec{s}, \vec{q}}_{\mix}(\T) \to L_2(\T)) \leq
\Big(\frac{C(d)}{n}\Big)^{\frac{s}{q\, (1+\log_2(d-1))}}\,,
\ee
where $C(d)$ is defined in \eqref{konstante}.
\\
{\rm (ii)} Let $d\ge 7$.
For all $ n$, $2 \le n \le e^{d-1}$ it holds
\[
a_n(I_d:H^{\vec{s},\vec{q}}_{\mix}(\T) \to L_2(\T)) \leq  \, n^{-\gamma^* s/q}\, ,
\]
where $\gamma^*$ is defined in \eqref{gammaa} (and $\kappa$ is given by $(d-1)/\ln n$).
\\
{\rm (iii)} Let $d\ge 2$.
For $n > 2$ we define
\[
\gamma (n,d) :=  \log_2\Big( 1+ \frac{2d}{\log_3 n} \Big) \, .
\]
For all $2 < n \le 3^d$ it holds
\[
a_n(I_d:H^{\vec{s},\vec{q}}_{\mix}(\T) \to L_2(\T)) \geq
    2^{-\frac{s}{q}}\,  n^{-\frac{s}{q \, \gamma (n,d)}}\,.
\]
\end{satz}

\begin{proof}
Part (i) is a direct consequence of Lemmas \ref{wichtig}, \ref{back} and Proposition
\ref{smallbbb}.
For the same reasons (ii) follows from Lemmas \ref{wichtig}, \ref{back} and Proposition \ref{smallbcb}.
The lower estimate in part (iii) is essentially proved in Krieg \cite{DK}.
One has to apply Theorem 4(ii) in \cite{DK} together with the observation
\[
 v:= \#\Big\{k \in \zz: ~ u_{{s},{q}} (k) = 2^{s/q}\Big\}= 2\, .
\]
This proves the claim.
\end{proof}

\begin{rem}
 \rm
(i) Roughly speaking, $a_n(I_d:H^{s, {q}}_{\mix}(\T) \to L_2(\T))$
behaves almost as $a_n(I_d:H^{s/{q},1}_{\mix}(\T) \to L_2(\T))$ in the preasymptotic range.
The index $q$ has a major impact on the approximation rate.
This fits quite well with our observation for the case $q=\infty$, see Theorem \ref{unendl}.
\\
(ii) The cases $q=1,2,2s$  have been discussed in Krieg \cite{DK} with almost the same outcome.
\end{rem}


\subsection{The preasymptotic decay in case of a non-constant $\vec{s}$}
\label{number2}


Now we continue with the discussion of the following situation.
There exists a natural number $\nu$, $1 \le \nu < d $, such that
\[
0<  s_1 =  s_2 = \ldots = s_\nu < t:= s_{\nu+1}  = \, \ldots \, = s_d\, .
\]
For technical reasons we distinguish into two cases: 
(i) $\nu \ge 5$ and (ii) $1 \le \nu\le 4$. We expect the following behavior. If the jump is large enough, 
then the influence of the variables $x_{\nu+1}, \ldots \, , x_d$
should almost disappear.

There is one more splitting.
As before we study $\vec{q} = \vec{1}$ first and continue with the general case afterwards.


\paragraph{The case $\vec{q} = \vec{1}$.}
\label{number21}


\noindent
{\bf Case (i):} Let $s_1 := 1$ and $\nu \ge 5$.
We choose
$\alpha := 1+  \log_2 (\nu-1) \ge 3$.
Lemma \ref{clever}, combined with \eqref{Konstante}, yields
\[
c(r,d)  \le    \, \Big( 2 \zeta (\alpha)-1  \Big)^{\nu-1}\,
\Big( 2 \zeta (\alpha t)-1  \Big)^{d-\nu}\, r^{\alpha} \, .
\]
It will be convenient to use the following modification of \eqref{wn-2} in case $u \ge 3$:
\beqq
2 \zeta (u)-1 & = & 1+ \frac{2}{2^u} \, \sum_{j=2}^\infty \Big(\frac{2}{j}\Big)^u \le
1+ \frac{2}{2^u} \, \sum_{j=2}^\infty \Big(\frac{2}{j}\Big)^3
\\
& \le &  1 + \frac{16}{2^u} (\zeta (3)-1)\, .
\eeqq
The particular value  $\zeta (3)$  is known with high precision and we have $\zeta (3) < 1.2021$.
Hence
\be\label{wn-2b}
2 \zeta (u)-1 < 1+ 3.2326 \, \cdot  \, 2^{-u}\, .
\ee
Since $\alpha \ge 3$ we may apply this inequality 
and therefore, using $1+x\le e^x$,
\[
\Big(2 \zeta (\alpha)-1\Big)^{\nu-1} \le  e^{\frac{3.2326}{2}} \, .
\]
Concerning the second factor we argue as follows
\[
\Big( 2 \zeta (\alpha t)-1  \Big)^{d-\nu} < \Big( 1+ \frac{3.2326}{2^{\alpha t}}\Big)^{d-\nu}
\le e^{\frac{3.2326 (d-\nu)}{2^{\alpha t}}}\le e^{3.2326}
\]
if $2^{\alpha t}\ge (d-\nu)$, i.e.,
\[
 \frac{\log_2 (d-\nu)}{1+ \log_2 (\nu-1)}\le t\, .
\]
Altogether this implies
\be\label{win8}
 c(r,d)  \le  e^{3\, \frac{3.2326}{2}} r^\alpha \le 38.02 \, r^\alpha\, .
\ee
Arguing as in the previous subsection including the switch from $s_1 = 1$ to $s_1 >0$
we obtain the following result.

\begin{satz}\label{Small1}
Let
$0 < s_1 =  s_2 = \ldots = s_\nu <  s_{\nu+1}  \le \ldots \, \le s_d$ for some $\nu \in \N$, $5 \le \nu < d$.
Suppose
\be\label{wn-4}
t:= \frac{s_{\nu+1}}{s_1} \ge \frac{\log_2 (d-\nu)}{1+ \log_2 (\nu-1)}\, .
\ee
For all $ n \in \N$ it holds
\[
a_n(I_d:H^{\vec{s},\vec{1}}_{\mix}(\T) \to L_2(\T)) \leq \Big(\frac{38.02}{n}\Big)^{\frac{s_1}{1+\log_2 (\nu-1)}}\,.
\]
\end{satz}

\begin{proof}
By means of Lemma  \ref{scale} the general case is reduced to the particular case
$1= t_1 =  t_2 = \ldots = t_\nu < t_{\nu+1}  \le  \, \ldots \, \le t_d$ with $t_i := s_i /s_1$, $i=1, \ldots \, , d$.
Now we observe that the embedding
\[
 H^{\vec{t},\vec{1}}_{\mix}(\T) \hookrightarrow  H^{t_1, \ldots \, , t_\nu, t_{\nu+1}, \ldots \, , t_{\nu+1},\vec{1}}_{\mix}(\T)
\]
has norm $1$. The multiplicativity of the approximation numbers yields that
\[
 a_n(I_d:H^{\vec{t},\vec{1}}_{\mix}(\T) \to L_2(\T)) \leq
a_n(I_d:H^{t_1, \ldots \, , t_\nu, t_{\nu+1}, \ldots \, , t_{nu+1},\vec{1}}_{\mix}(\T) \to L_2(\T)) \, ,
\]
see Lemma \ref{back}.
Concerning the numbers on the right-hand side we may apply the estimate \eqref{win8}.
Arguing as several times before this proves
\eqref{wn-4} with $\vec{s}$ replaced by $\vec{t}$.
An application of Lemma \ref{scale} with $\lambda = s_1$ completes the proof.
\end{proof}

\begin{rem}
 \rm
(i)
 In  Theorem \ref{Small1} we essentially  consider the case of two different smoothness levels.
 If the jump from $s_\nu$ to $s_{\nu +1}$  is large enough, see  \eqref{wn-4}, then we get an estimate of $a_n$ in which the
 influence of the variables $x_{\nu + 1}, \ldots \, x_d$ has disappeared.
\\
(ii) Our method works as well without the restriction \eqref{wn-4}.
However, we believe, that the case of a big jump is more interesting as that one of a small jump
because in the latter case we are approaching the case of a constant smoothness vector again.
So we skip the details for small jumps here.
\end{rem}

\noindent
{\bf Case (ii):} Let $s_1=1$ and  $1 \le \nu \le 4$.
We choose $\alpha = 2$ in Lemma \ref{clever} and obtain
\beqq
c(r,d)  & \le &     \, \Big( 2 \zeta (2)-1  \Big)^{\nu-1}\,
\Big( 2 \zeta (2 t)-1  \Big)^{d-\nu}\, r^{2}
\\
 & \le & \Big(\frac{\pi^2}{3} - 1\Big)^{\nu-1}\,
\Big( 1 + 2^{-2t}\Big( 2+  \frac{4}{2t-1}\Big)  \Big)^{d-\nu}\, r^{2}
\, ,
\eeqq
see \eqref{wn-2}.
If $t\ge \max \Big(\frac 32, ~ \frac{\log_2 (d-\nu)}{2}\Big)$,
then
\[
\Big( 2 \zeta (2 t)-1  \Big)^{d-\nu}\le
\Big( 1 + \frac{4}{2^{2t}}  \Big)^{d-\nu}
\le \Big( 1 + \frac{4}{d-\nu}  \Big)^{d-\nu} \le e^4\, .
\]
This results in
\[
 c(r,d)\le e^4 \, \Big(\frac{\pi^2}{3} - 1\Big)^{\nu-1}\, r^2\, .
\]
Arguing as in prove of Theorem \ref{Small1} we get the following.

\begin{satz}\label{Small2}
Let
$0 < s_1 =  s_2 = \ldots = s_\nu <  s_{\nu+1}  \le \ldots \, \le s_d$ for some $\nu \in \N$, $1 \le \nu < \min (5,d)$.
Suppose
\[
t:= \frac{s_{\nu+1}}{s_1} \ge \max \Big(\frac 32, ~
\frac{\log_2 (d-\nu)}{2}\Big)\, .
\]
For all $ n \in \N$ it holds
\[
a_n(I_d:H^{\vec{s},\vec{1}}_{\mix}(\T) \to L_2(\T)) \leq
\Big[e^4 \, \Big(\frac{\pi^2}{3} - 1\Big)^{\nu-1}\,\frac{1}{n}\Big]^{\frac{s_1}{2}}\,.
\]
\end{satz}

\begin{rem}
 \rm
{\rm (i)}
As in the previous Theorem we can make the following observation:
if the jump between $s_\nu$ and $s_{\nu +1}$ is large enough, then the influence
of the variables $x_{\nu+1}, \ldots \, , x_d$ disappears.
\\
{\rm (ii)} We believe that the exponent $s_1/2$ can be improved.
\end{rem}

For later use we investigate one more case. Let
\[\nu = 1,\quad  d\ge 5, \quad s_1 = 1 \qquad \mbox{and}\qquad
1 < s_2 \le \ldots \le s_d \, .
\]
We choose $\alpha := 1+ \log_2 (d-1)\ge 3$.
Applying  Lemma \ref{clever}, \eqref{wn-2b} and $A_\alpha = 1$ we conclude that

\beqq
c(r,d) & \le &  \prod_{j=2}^d (2 \zeta (\alpha s_j)-1) \, r^\alpha
\le  (1+ 3.2326 \, \cdot  \, 2^{-\alpha s_2})^{d-1} \, r^\alpha
\\
& = & \Big(1+ 3.2326 \, \cdot  \, 2^{-s_2}\,  \frac{1}{(d-1)^{s_2}}\Big)^{d-1} \, r^\alpha
\\
& \le & e^{\frac{3.2326}{2^{s_2} (d-1)^{s_2-1}}} \, r^\alpha\, .
\eeqq
We define
\be
\label{konstante2}
C(t,d):= e^{\delta (t,d)}\, , \qquad \delta (t,d):= \frac{3.2326}{2^{t} (d-1)^{t-1}}\, .
\ee
As a consequence we get the following.

\begin{satz}\label{Smalle4}
Let $d\ge 5$ and  $0 <s_1 <  s_2   \le \,  \ldots \,\le s_d$.
For all $ n \in \N$ it holds
\[
a_n(I_d:H^{\vec{s},\vec{1}}_{\mix}(\T) \to L_2(\T)) \leq
\Big[\frac{C(s_2/s_1, d)}{n}\Big]^{\frac{s_1}{1+\log_2 (d-1)}}\,,
\]
where $C(s_2/s_1, d)$ is defined in \eqref{konstante2}.
\end{satz}

\begin{rem}
 \rm
Of course, if $d\to \infty$, then $C(t, d) \to 1$.
In particular, if $s_2/s_1 \ge 2$, then
$1< C(s_2/s_1, d)\le e^{1/4}$.
In such a situation the influence of $s_2, \ldots \, , s_d$ and therefore of
$x_2, \ldots \, , x_d$ is rather weak.
\end{rem}


\paragraph{The case $\vec{q} \neq  \vec{1}$.}

Here we will proceed as in the previous paragraph.
Theorem \ref{Small1} in combination with Lemmas \ref{wichtig}, \ref{back} yield the following.

\begin{cor}\label{Smalle1}
Let
$0< s_1 =  s_2 = \ldots = s_\nu <  s_{\nu+1}  \le \ldots \, \le s_d$ for some $\nu \in \N$, $5 \le \nu < d$.
Let $\vec{q}$ be a constant vector generated by some $q>1$.
Suppose
\[
t:= \frac{s_{\nu+1}}{s_1} \ge \frac{\log_2 (d-\nu)}{1+ \log_2 (\nu-1)}\, .
\]
For all $ n \in \N$ it holds
\[
a_n(I_d:H^{\vec{s},\vec{q}}_{\mix}(\T) \to L_2(\T)) \leq \Big(\frac{38.02}{n}\Big)^{\frac{s_1}{q(1+\log_2 (\nu-1))}}\,.
\]
\end{cor}

Similarly, Theorem \ref{Small2}, Theorem \ref{Smalle4} and  Lemmas \ref{wichtig}, \ref{back} can be used to derive the next two results.

\begin{cor}\label{Smalle2}
Let
$s_1 =  s_2 = \ldots = s_\nu <  s_{\nu+1}  \le \ldots \, \le s_d$ for some $\nu \in \N$, $1 \le \nu < \min (5,d)$.
Let $\vec{q}$ be a constant vector generated by some $q>1$.
Suppose
\[
t:= \frac{s_{\nu+1}}{s_1} \ge \max \Big(\frac 32, ~\frac{\log_2 (d-\nu)}{2}\Big)\, .
\]
For all $ n \in \N$ it holds
\[
a_n(I_d:H^{\vec{s},\vec{q}}_{\mix}(\T) \to L_2(\T)) \leq
\Big[e^4 \, \Big(\frac{\pi^2}{3} - 1\Big)^{\nu-1}\,\frac{1}{n}\Big]^{\frac{s_1}{2q}}\,.
\]
\end{cor}

\begin{cor}\label{Smalle41}
Let $d\ge 5$ and  $0 <s_1 <  s_2 \le \,  \ldots \,  \le s_d$.
Let $\vec{q}$ be a constant vector generated by some $q>1$.
For all $ n \in \N$ it holds
\[
a_n(I_d:H^{\vec{s},\vec{q}}_{\mix}(\T) \to L_2(\T)) \leq
\Big[\frac{C(s_2/s_1, d)}{n}\Big]^{\frac{s_1}{q( 1+\log_2 (d-1))}}\,,
\]
where $C(s_2/s_1, d)$ is defined in \eqref{konstante2}.
\end{cor}

Also the case of non-constant vector $\vec{q}$ can be treated.
But here one has to take into account that a renumbering of the variables will not influence the approximation numbers of an embedding into $L_2 (\T)$.
Again our argument will be based on the embedding $H^{\vec{s},\vec{q}}_{\mix}(\T) \hookrightarrow
H^{\vec{s}/\vec{q}, \vec{1}}_{\mix}(\T)$, see Lemma \ref{wichtig}.
Looking at the sequence $s_1/q_1, \ldots \, , s_d/q_d$, in general, they will have no ordering.
Let us denote by $r_1, r_2, \ldots \, , r_d$ the rearranged sequence with
\[
r_1 := \min \Big\{ \frac{s_j}{q_j}:~ j=1, \ldots \, , d\Big\}\, .
\]
In addition we need the counterpart of $\nu$, defined as
\[
 \mu := \# \Big\{j \in \{1, \ldots \, , d\}:   \frac{s_j}{q_j} = r_1\Big\}\, .
\]
For simplicity we only consider the generalization of Corollary \ref{Smalle1}.

\begin{cor}\label{Smalle3}
Let
$0< s_1 \le   s_2  \le \ldots \, \le s_d$.
Let $\vec{q}$ be a finite vector such that  $\min_{j=1, \ldots \, , d} q_j >1$.
Suppose  $5 \le \mu < d$ and
\[
t:= \frac{r_{\mu+1}}{r_1} \ge \frac{\log_2 (d-\mu)}{1+ \log_2 (\mu-1)}\, .
\]
For all $ n \in \N$ it holds
\[
a_n(I_d:H^{\vec{s},\vec{q}}_{\mix}(\T) \to L_2(\T)) \leq \Big(\frac{38.02}{n}\Big)^{\frac{r_1}{q(1+\log_2 (\mu-1))}}\,.
\]
\end{cor}


\subsection{The preasymptotic decay for logarithmically growing $\vec{s}$}
\label{numberc}


In this subsection we consider vectors $\vec{s}$ with strongly increasing components, more exactly, we shall investigate
\[
 s_j := 1+ \beta \, \log_2 j\, , \qquad j \in \N\, .
\]
Here $\beta>0$ will be chosen later on.
This time we will not try to determine the optimal rate, we will be satisfied with an estimate
completely independent of $d$.


\paragraph{The case $\vec{q} = \vec{1}$.}


Let $\alpha$ be any number $>\frac{1}{\beta}$.
Lemma \ref{clever}, combined with \eqref{wn-2} and the trivial inequality
$1+x\le e^x$, $x \ge 0$, yields
\beqq
c(r,d) & \le &  A_\alpha \, r^\alpha\,  \prod_{j=2}^d \Big( 2 \zeta (\alpha s_j)-1  \Big)
\le  A_\alpha\,  r^{\alpha}\, \prod_{j=2}^d \Big( 2 \zeta (\alpha + \alpha \, \beta \, \log_2 j)-1  \Big)
\\
& \le & A_\alpha \, r^{\alpha} \, \prod_{j=2}^d \Big(1 + 2^{-(\alpha + \alpha\, \beta \,  \log_2 j)}\, \Big( 2+ \frac{4}{\alpha + \alpha \beta \, \log_2 j -1} \Big)\Big)
\\
& \le & A_\alpha \, r^{\alpha} \,  e^{(6/2^{\alpha}) \, \sum_{j=2}^d \frac{1}{j^{\alpha\, \beta}}} 
\, .
\eeqq
We observe that
\[
\sum_{j=2}^d  \frac{1}{j^{\alpha\, \beta}}  =   (\zeta (\alpha\, \beta)-1) <\infty
\] 
because of $ \alpha \, \beta >1$.
For brevity we put
\be\label{konstante4}
C_{\alpha,\beta} := \frac{6}{2^{1/\beta}}\,  (\zeta (\alpha\, \beta)-1) \, .
\ee
This implies the following.

\begin{satz}\label{smallcdf}
Let $d\ge 2$ and suppose
\[
 s_j \ge (1+ \beta \, \log_2 j) \, s_1\, , \qquad j \in \N\, ,
\]
for some $\beta >0$. Let $\alpha > 1/\beta$.
Then it holds
\be\label{ws-46efg}
a_n(I_d:H^{\vec{s},\vec{1}}_{\mix}(\T) \to L_2(\T)) \leq \Big(\frac{A_\alpha \, e^{C_{\alpha,\beta}}}{n}\Big)^{\frac{s_1}{\alpha}}
\ee
for all $n\in \N$. Here $C_{\alpha,\beta}$ is defined in \eqref{konstante4}.
\end{satz}


\paragraph{The case $\vec{q} \neq \vec{1}$.}


Again we shall work with Lemmas \ref{wichtig}, \ref{back}.
Then, as a consequence of  Theorem \ref{smallcdf} we find the following generalization.

\begin{cor}\label{smallecd2}
Let $d\ge 2$ and suppose
\[
 s_j \ge (1+ \beta \, \log_2 j) \, s_1\, , \qquad j \in \N\, ,
\]
for some $\beta >0$. Let $\alpha > 1/\beta$.
Let $\vec{q}$ be a constant vector generated by a finite $q>1$.
Then it holds
\[
a_n(I_d:H^{\vec{s},\vec{q}}_{\mix}(\T) \to L_2(\T)) \leq \Big(\frac{A_\alpha \,  e^{C_{\alpha,\beta}}}{n}\Big)^{\frac{s_1}{\alpha q}}
\]
for all $n\in \N$. Here $C_{\alpha,\beta}$ is defined in \eqref{konstante4}.
\end{cor}

\begin{rem}
 \rm
Our estimates in this subsection result in polynomial error bounds independent of $d$.
In other words, if the smoothness components grow moderately then $d$ does not influence the error bounds. 
To be more precise, like in the slightly different situation discussed in \cite{PW}, we 
also have here {\em strong polynomial tractability} if and only if 
\begin{equation}\label{limsup}
	\limsup\limits_{j\to\infty} \frac{\ln j}{s_j} < \infty\,.
\end{equation}
We restrict ourselves to the case $\vec{q} = \vec{1}$. Indeed, according to the fundamental 
result in \cite[Thm.\ 5.1]{NoWo08} strong polynomial tractability holds if and only if
$$
	\sup\limits_{d} \sum\limits_{n=1}^{\infty} a_{n,d}^{2\tau} < \infty\quad\text{for some } \tau >0\,,
$$
where $a_{n,d} = a_{n,d}(I_d)$ denotes the $n$-th approximation number of the respective $d$-dimensional problem. Then the index $p$ on the bottom of page 5 above (equivalence result in \cite{PW}) can be chosen 
as $1/2 \tau$. One is therefore interested in small $\tau$'s. In our specific situation ($\vec{q} = \vec{1}$) 
this condition can be rephrased as 
$$
	\sup\limits_{d\in\N} \sum\limits_{k} \prod\limits_{j=1}^d (1+|k_j|)^{-2s_j\tau} = 
	\sup\limits_{d\in\N}\prod\limits_{j=1}^d (2\zeta(2\tau s_j)-1)  <\infty
$$
for some $\tau>0$ such that $\delta := \inf\limits_{i}2\tau s_i > 1$. By our estimate in \eqref{wn-2} we have  
$$
	1+2^{-t} \leq 2\zeta(t)-1 \leq 1+2^{-t}(2+4/(t-1))
$$
and hence
\begin{equation}\label{sum}
				\sup\limits_{d}\prod\limits_{j=1}^d (2\zeta(2\tau s_j)-1) < \infty \Longleftrightarrow \sum\limits_{j = 1}^{\infty} 2^{-2\tau s_j}<\infty\,.
\end{equation}
We are now in the same situation as discussed in \cite[p. 416]{PW}. There it is proved that the finiteness of the sum 
in \eqref{sum} for some $\tau >0$ is equivalent to \eqref{limsup}. The optimal tractability index is then $p^* = 1/2\tau^*$ 
with $\tau^*$ being the infimum over all $\tau$ such that \eqref{sum} is finite. This corresponds to our decay rates above. 
Note that we additionally give precise numerical values for the involved constants in the error bounds. 
Note also that in case  $\beta \downarrow 0$ the rate $s_1/\alpha \to 0$ which represents another illustration of the above 
characterization of polynomial tractability given in \cite{PW}. 
\end{rem}


\section{Embeddings into the energy space}
\label{mixediso}


In this last section we shall investigate the preasymptotic behavior of the approximation numbers of the embeddings 
$$I_d:~H^s_\mix (\T) \to H^1(\T)\qquad(s>1,d\in\N)\,.$$
Here we use the short notation $H^s_\mix (\T)$ for the space $H^{\vec{s},\vec{q}}_\mix (\T)$ with
constant smoothness vector $\vec{s}=(s,\ldots,s)$ and constant vector $\vec{q}=(2,\ldots,2)$. The space
$H^1(\tor^d)$ belongs to the scale of isotropic Sobolev spaces $H^s(\tor^d)$ of fractional order $s>0$, endowed with the norm
$$
	\|f|H^s(\tor^d)\| := \Big[\sum\limits_{k\in \Z}|c_k(f)|^2
	\Big(1+\sum\limits_{j=1}^d|k_j|^2\Big)^s\Big]^{1/2}.
$$
The asymptotic order of the decay of the approximation numbers is well-known, see, e.g., \cite{GrKn08}, \cite{DiUl13}, \cite{ByDiWiUl14}. It holds
\begin{equation}\label{asym}
	c_s(d)\,n^{-(s-1)}\le a_n(I_d) \le C_s(d)\, n^{-(s-1)}
\end{equation}
for all $n\in\N$, with constants depending on $s$ and $d$ but not on $n$. 
To the best of our knowledge, the only preasymptotic estimate so far is 
given in \cite{KMU}. We will use a different technique here. By analogous arguments as in Subsection \ref{sing_diag} we see that
$$
	a_n(I_d:H^{s}_{\mix}(\tor^d) \to L_2(\tor^d)) = a_n(D_w:\ell_2(\Z)\to \ell_2(\Z)),
$$
where the diagonal operator $D_w:\ell_2(\Z) \to \ell_2(\Z)$ is defined via the weight
\begin{equation}\label{f3}
	w(k) :=  \frac{\Big(1+\sum_{j=1}^d |k_j|^2\Big)^{1/2}}{\prod\limits_{j=1}^d (1+|k_j|^2)^{s/2}}\quad,\quad k\in \Z\,.
\end{equation}
In order to find preasymptotic estimates for the decay of the non-increasing rearrangement of this weight we use a point-wise larger weight, namely
\begin{equation}
	w(k) \leq   \frac{\prod\limits_{j=1}^d (1+|k_j|^2)^{1/2}}{\prod\limits_{j=1}^d (1+|k_j|^2)^{s/2}} = \prod\limits_{j=1}^d (1+|k_j|^2)^{-(s-1)/2} =: \tilde{w}(k)\,.
\end{equation}
Then we have
\begin{equation}\label{H1vsH1mix}
	a_n(D_w) \leq a_n(D_{\tilde{w}}) = a_n(id: H^{s-1}_\mix (\T) \to L_2(\T))\,.
\end{equation}
Applying Theorem \ref{smalldd} we obtain the following result.

\begin{prop}\label{main0} Let $s>1$, $d\geq 3$ and $n\ge 2$. Then
\be\label{w-neu}
a_n(I_d:H^{s}_{\mix}(\T) \to H^1(\T)) \leq
\Big(\frac{C(d)}{n}\Big)^{\frac{s-1}{2\, (1+\log_2(d-1))}}\,,
\ee
where $C(d)$ is defined in \eqref{konstante} and ranges in the interval $[e ,6.25]$ depending on $d$.
\end{prop}

\begin{rem}\rm (i) This result is already an improvement over the one in \cite[(8)]{KMU}. 
There two of the authors and S. Mayer gave a similar estimate with a slightly worse exponent, 
the constant $C = e^2$ and the range $1\leq n \leq 4^d$.

(ii) Proposition \ref{main0} directly implies that the corresponding family of approximation problems 
$(I_d:H^{s}_{\mix}(\T) \to H^1(\T))_{d\in\N}$
 is \emph{quasi-polynomially tractable} in the sense of Gnewuch and Wo{\'z}niakowski \cite{GnWo11}. See also \cite{DiUl13} and the comment after formula (6) in \cite{KMU}.
\end{rem}

\noindent This result can be further improved by using a different technique which does not 
require the rearrangement of $(w(k))_{k\in \Z}$. Note that this technique is not new, 
see for instance \cite[Thm.\ 5.1 and (5.2)]{NoWo08}\,.

\begin{prop}\label{main1} Let $s>1$, $d\geq 4$ and $n\geq 8$.Then 
$$
	a_n(I_d:H^{s}_{\mix}(\tor^d) \to H^1(\tor^d)) \leq
	\Big(\frac{e^2}{n}\Big)^{\frac{s-1}{2\log_2 d}}\,.
$$
\end{prop}

\begin{proof} 
The main idea is to compare (quasi-)norms in Schatten $p$-classes. Let $p:= \frac{\log_2 d}{s-1}$, and let $w(k)$ and $\tilde{w}(k)$ be as before. Then
\beqq
n a_n (I_d)^{2p} & \le\sum \limits^\infty_{n=1} a_n (I_d)^{2p} = \sum \limits^\infty_{n=1} a_n(D_w)^{2p} \le\sum \limits^\infty_{n=1} a_n (D_{\tilde{w}})^{2p}
\\
&  = 
\sum \limits_{k \in \zz^d} \tilde{w}(k)^{2p} = \left(1+2 \sum \limits^\infty_{k=1} (1+k^2)^{-(s-1)p}\right)^d\,.
\eeqq
Now setting $x : = \frac{2}{2^{(s-1)p}} = \frac{2}{d}$, we have $\frac{x^2}{2} = \frac{2}{4^{(s-1)p}}$ and
\beqq
1 + 2 \sum \limits^\infty_{k=1} \frac{1}{(1+k^2)^{(s-1)p}} & = & 1 + \frac{2}{2^{(s-1)p}} + \frac{2}{4^{(s-1)p}} \sum \limits^\infty_{k=2} \left(\frac{4}{1+k^2}\right)^{(s-1)p}
\\
& \le & 1 + x + \frac{x^2}{2} \cdot \sum \limits^\infty_{k=2} \left(\frac{4}{1+k^2}\right)^2 \,, \mbox{ since } \, (s-1)p = \log_2 d \ge 2\,.
\eeqq
Moreover,
\beqq
\sum \limits^\infty_{k=2} \left(\frac{4}{1+k^2}\right)^2 & = & \left(\frac{4}{5}\right)^2 + \left(\frac{4}{10}\right)^2 + \sum \limits^\infty_{k=4} \left(\frac{4}{k^2+1}\right)^2
\\
& \le & \frac{4}{5} + 16 \int \limits^\infty_3 \frac{dx}{x^4} = \frac{4}{5} + \frac{16}{81} \le 1\,.
\eeqq
Altogether this implies
$$
n a_n (I_d)^{2p} \le \left(1 + x + \frac{x^2}{2}\right)^d \le e^{xd} = e^2\,,
$$
and therefore
$$
a_n (I_d) \le \left(\frac{e^2}{n}\right)^{\frac{1}{2p}} = \left(\frac{e^2}{n}\right)^{\frac{s-1}{2 \log_2 d}}\,.
$$
\end{proof}

\begin{rem}\rm We have actually proved a stronger result, namely that 
$$
 	a_n(I_d:H^{s}_{\mix}(\tor^d) \to H^1_{\mix}(\tor^d)) \leq
	\Big(\frac{e^2}{n}\Big)^{\frac{s-1}{2\log_2 d}}\,,
$$
which, in turn, improves on Theorem \ref{smalldd} in case $q=2$.
\end{rem}

The following proposition will finally use the specific structure of the energy space in terms of a different estimate for $w(k)$ and gives an improvement over Proposition \ref{main1} in some situations (see Remark \ref{finalrem} below).

\begin{prop}\label{main2} Let $s>1$ and $d\in \N$ such that $d\geq 1+\max\{2^{s-1}, 2^{1/(s-1)}\}$. Then, with $$
	C(d) = e\,(2.154 + 3/d)
$$
we have for all $n\in\N$
$$
	a_n(I_d:H^{s}_{\mix}(\tor^d) \to H^1(\tor^d))
	\leq \sqrt{d}\Big(\frac{C(d)}{n}\Big)^{\frac{s}{2(1+\log_2(d-1))}}\,.
$$
\end{prop}
\begin{proof} 
For any $p \ge 1$, H\"older's inequality implies
\beqq
\left(1+\sum \limits^d_{j=1} |k_j|^2\right)^p & = & \left(\left(1+|k_1|^2\right) + |k_2|^2 + \dots + |k_d|^2\right)^p
\\
& \le & d^{p-1} \left[\left(1+|k_1|^2\right)^p + |k_2|^{2p} + \dots + |k_d|^{2p}\right]\,,
\eeqq
whence, with the weight $w$ as given in (\ref{f3}),
\beqq
n a_n (I_d)^{2p} & \le &  \sum \limits^\infty_{n=1} a_n (I_d)^{2p} 
 =  \sum \limits^\infty_{n=1} a_n (D_w)^{2p} = 
\sum \limits_{k \in \zz^d} \frac{\Big(1+\sum^d_{\ell=1} |k_\ell|^2\Big)^p}{\prod^d_{j=1} (1+|k_j|^2)^{sp}}
\\
& \le &  d^{p-1} \left[\sum \limits_{k \in \zz^d} \frac{(1+|k_1|^2)^p} {\prod^d_{j=1} (1+|k_j|^2)^{sp}}
+\sum_{\ell=2}^d\sum \limits_{k \in \zz^d} \frac{|k_\ell|^{2p}} {\prod^d_{j=1} (1+|k_j|^2)^{sp}}\right]\,.
\eeqq
Note that the $d-1$ sums over $\ell = 2,...,d$ in the last line are all equal. Setting
$$
A:=\sum \limits_{k \in \zz} \frac{1}{(1+|k|^2)^{sp}}\quad,\quad
B:=\sum \limits_{k\in \zz} \frac{1}{(1+|k|^2)^{(s-1)p}}\quad,\quad
C:=\sum \limits_{k \in \zz} \frac{|k|^{2p}}{(|1+|k|^2)^{sp}}
$$
and evaluating the sums over $k\in\zz^d$ coordinate by coordinate we obtain
$$
n a_n (I_d)^{2p} \le d^{p-1} A^{d-1} (B+(d-1)C)\,.
$$ 
It remains to estimate $A,B,C$. We choose now 
$$
p := \frac{1+\log_2(d-1)}{s}\,.
$$ 
The above estimates required $p \ge 1$, which is equivalent to our assumption $d-1\ge 2^{s-1}$.
For $d\ge 3$ we have $sp\ge 1+\log_2(d-1)\ge 2$, and similarly as in the previous proof we obtain, now with
$x := \frac{2}{2^{sp}} = \frac{1}{d-1}$, 
\beqq
A & = & 1 + \underbrace{\frac{2}{2^{sp}}}_{=x} + 
\underbrace{\frac{2}{4^{sp}}}_{=\frac{x^2}{2}}\underbrace{\left(\left(\frac{4}{5}\right)^{sp} + 
\left(\frac{4}{10}\right)^{sp} + \dots \right)}_{\le (\frac{4}{5})^2 + (\frac{4}{10})^2 + \dots \le 1}
\\
& \le & 1 + x + \frac{x^2}{2} \le e^x=e^{1/(d-1)}\,, \quad\mbox{ whence } \, A^{d-1} \le e\,.
\eeqq
In the estimate of the term $B+(d-1)C$ we use that $(s-1)p \ge 1$, which is equivalent to 
$1 + \log_2 (d-1) \ge \frac{s}{s-1} = 1 + \frac{1}{s-1}$, i.e. equivalent to $d-1 \ge 2^{\frac{1}{s-1}}$.
We get 
\beqq
B + (d-1) C & = & 1 + 2 \sum \limits^\infty_{k=1} \frac{1}{(1+k^2)^{(s-1)p}} + 2 (d-1) \sum \limits^\infty_{k=1} \frac{k^{2p}}{(1+k^2)^{sp}}
\\
& \le & 1 + \underbrace{\frac{2}{2^{(s-1)p}}}_{\le 1}  + 2 \sum \limits^\infty_{k=2} \frac{1}{1+k^2} + 
\underbrace{\frac{2 (d-1)}{{2^{sp}}}}_{=1} +2(d-1) \sum \limits^\infty_{k=2} \frac{1}{1+k^2}
\\
& \le & 3 + 2d \sum \limits^\infty_{k=2}  \frac{1}{1+k^2} \le d \Big(\frac{3}{d} + 2.154\Big)\,.
\eeqq
Here we used
\beqq
\sum \limits^\infty_{k=1} \frac{1}{1+k^2} 
= \frac\pi 2\cdot\coth(\pi) = \frac\pi 2\cdot 1.0037418\dots\le 1.577\,.
\eeqq
This proves the desired estimate
$$n a_n(I_d)^{2p} \le d^p \cdot \underbrace{e \cdot (2.154+3/d)}_{=C(d)}\quad,\quad\text{hence}\quad 
a_n(I_d) \le \sqrt{d} \left(\frac{C(d)}{n}\right)^{\frac{s}{2(1+\log_2(d-1))}}\,.$$
\end{proof}

\begin{rem}\label{finalrem}\rm For $s=2$, Propositions \ref{main1} and \ref{main2} give 
the following estimates:
\begin{itemize}
\item[(i)] $a_n\left(I_d : H^2_{\mix} (\tor^d) \rightarrow H^1(\tor^d)\right) \le \left(\frac{e^2}{n}\right)^{\frac{1}{2 \log_2d}} \qquad\qquad (d \ge 4)$ ,
\item[(ii)] $a_n \left(I_d : H^2_{\mix} (\tor^d) \rightarrow H^1(\tor^d)\right)\le 
\sqrt{d} \left(\frac{e\,(2.154 + 3/d)}{n}\right)^{\frac{1}{1 + \log_2(d-1)}} \quad (d \ge 3)$\,.
\end{itemize}
Clearly, in large dimensions and for (moderate) $n$ in the preasymptotic range $n\le 2^d$, the second bound is better.
\end{rem}

\nocite{KUV,GrOs15}

\paragraph{Acknowledgment.} The authors would like to thank Aicke Hinrichs, Lutz K\"ammerer, David Krieg, 
Peter Oswald, Daniel Potts, Klaus Ritter and Henryk Wo{\'z}niakowski for several fruitful discussions on the topic. 
They particularly thank Kateryna Pozharska for reading the whole manuscript and making several useful suggestions 
to improve the presentation. Finally we have to thank two anonymous reviewers for valuable hints to improve the 
manuscript. T.U.\ would like to acknowledge support by the DFG Ul-403/2-1.

\bibliographystyle{abbrv}

\end{document}